\documentclass[12pt,a4paper]{amsart}
\usepackage{amsfonts}
\usepackage{amsthm}
\usepackage{amsmath}
\usepackage{amscd}
\usepackage[latin2]{inputenc}
\usepackage{t1enc}
\usepackage[mathscr]{eucal}
\usepackage{indentfirst}
\usepackage{graphicx}
\usepackage{graphics}
\usepackage{pict2e}
\usepackage{epic}
\usepackage{enumitem}
\numberwithin{equation}{section}
\usepackage[margin=1.5in]{geometry}
\usepackage{epstopdf} 
\usepackage{float}
\usepackage{setspace}
\theoremstyle{plain}
\newtheorem{theorem}{Theorem}[section]
\newtheorem{lemma}[theorem]{Lemma}
\newtheorem{corollary}[theorem]{Corollary}

\theoremstyle{definition}
\newtheorem{definition}{Definition}[section]

\newtheorem{example}{Example}[section]
\newtheorem{innercustomthm}{Theorem}
\newenvironment{customthm}[1]
  {\renewcommand\theinnercustomthm{#1}\innercustomthm}
  {\endinnercustomthm}

\newcommand{\Hom}{{\rm{Hom}}}
\newcommand{\Ext}{{\rm{Ext}}}
\newcommand{\Tor}{{\rm{Tor}}}
\newcommand{\Ker}{{\rm{Ker}}}
\newcommand{\Aut}{{\rm{Aut}}}
\newcommand{\Q}{\mathbb{Q}}
\newcommand{\Z}{\mathbb{Z}}
\newcommand{\power}{\mathcal{P}}
\DeclareMathOperator*{\colim}{colim}
\DeclareMathOperator*{\aut}{Aut}
\newcommand{\og}{\mathscr{O}_G}
\newcommand{\cg}{\mathscr{C}_G}

\usepackage{pgfplots}

\begin{document}

\title{Bredon Cohomology of Polyhedral Products}
\author[Q. Zhu]{Qiaofeng Zhu}
\address{University of Rochester \\ Department of Mathematics \\
Rochester NY 14627} 
\email{q.zhu@rochester.edu}

\begin{abstract} 
A polyhedral product is a natural subspace of a Cartesian product, which is
specified by a simplicial complex $K$. The automorphism group $\Aut(K)$ of $K$ induces a group action on the polyhedral product. In this paper we study this group action and give a formula for the fixed point set of the polyhedral product for any subgroup $H$ of $\Aut(K)$. We use the fixed point data to compute examples of Bredon cohomolohgy for small non-Abelian groups such as $D_8$ and $\Sigma_4$.
 
\end{abstract}

\maketitle

\setcounter{tocdepth}{2}
\tableofcontents
\section{Introduction}
The category of Bredon coefficient systems over a commutative ring $R$, denote by $\cg^R$, is the category of contravariant functors from the canonical orbit category of $G$ to the category of $R$-modules. It is an Abelian category with enough injectives. Using this category $\cg$, Bredon  defined a homology and a cohomology theory for $G$-spaces \cite{Bredon}. On the other hand, subspaces of products have been studied recently which give a natural setting for Bredon cohomology. One of the purposes of this paper is to calculate Bredon cohomology of those subspaces for small non-abelian group $G$ such as $D_8$, the dihedral group of order 8, and $\Sigma_4$, the symmetric group of order 24. 

Through this paper we will assume that $G$ is a finite group. Given a $G$-space $X$, in order to compute Bredon cohomology, we need to access the homology of the fixed points sets $X^H$ for all subgroups $H$ of $G$ as well as an injective resolution for a Bredon coefficient system. Unfortunately, both parts are not easy to obtain in general. Much recent researches, such as \cite{BSW17} and \cite{Zeng17}, focus on $G=C_{p^n}$, the cyclic group of prime power order, and on a more rigid subcategory of $\cg$ such as category of Mackey functors. 

On the other hand, given a simplicial complex $K$ on $m$ vertices, the group $\Aut(K)$ is a subgroup of $\Sigma_m$. It naturally acts on the polyhedral product $Z(K,(X,A))$ by permuting coordinates. Polyhedral products and this group action are defined in section 2. In that section, we also defined the strong quotient of a simplicial complex $K$ with respect to a subgroup $G<\Aut(K)$, denoted by $K//G$, which is another simplicial complex with vertex set $V_G$. This construction provides a formula for the fixed point set of $G$. One of the main results of this paper is the following fixed point set theorem. Let $O_G$ be the set of orbits of vertices of $K$ under the group action $G$. And denote the number of vertex obrits that do not contained in $K$ by $k_G=|O_G-K|$. 

\begin{customthm}{\ref{fixed point theorem}}
Let $K$ be an abstract simplicial complex with vertex set $[m]$ and $(X, A)$ a based CW-pair. For any subgroup $G< \aut(K)$, $G$ acts on $Z(K; (X,A))$. T. Then 
$$Z(K; (X, A))^G=Z(K//G; (X,A))\times A^{k_G}.$$
\end{customthm}

Meanwhile, Define the Weyl group $WH$ for any subgroup $H$ of $G$ by $$WH=N_G(H)/H.$$
For any $WH$-module $V_H$over $\Q$, Doman in \cite{Doman88}, construct an injective Bredon coefficient system $\underline{I(V_H)}$. In addition, he provide an injective envelope for any Bredon coefficent system over $\Q$ using those injective coefficient system.

\begin{customthm}{\ref{injective envelope}}[\cite{Doman88}]
$f:\underline{M}\rightarrow\oplus \underline{I(V_H)}$ is a injective envelope of $\underline{M}$, where direct sum is over all conjugate classes of $G$.
\end{customthm}

Details of this theorem are explained in Section 4. With the help of this theorem, given any Bredon coefficient system over $\Q$, there is a general method to write down its injective resolution. 

One of the tools for computing Bredon cohomology is the universal coefficient spectral sequences.

\begin{customthm}{\ref{universal coefficient spectral sequence}}[\cite{Bredon}, \cite{May96}]
There is a universal coefficient spectral sequence that converges to Bredon cohomology 
$$E_2^{p,q}=\Ext_{\cg^R}^{p,q}(\underline{H_*}(X), \underline{M})\Rightarrow H_G^n(X, \underline{M}),$$
and a universal coefficient spectral sequence that converges to Bredon homology 
$$E^2_{p,q}=\Tor^{\cg^R}_{p,q}(\underline{H_*}(X), \underline{N})\Rightarrow H^G_n(X, \underline{M}).$$
\end{customthm}

Combing all the results above, it is possible to compute Bredon cohomology of polyhedral products over rationals using universal coefficient spectral sequences, even for a non-Abelian group. In section 4, we will see examples of such computations for $D_8$ and $\Sigma_4$.

In particular, for the constant $\underline{\Q}$ coefficient system, we have the following result for Bredon cohomology

\begin{customthm}{\ref{BCforconstantQ}}
Let $K$ be the boundary of 3-simplex,  the polyhedral product $Z(K;(D^1, S^0))$, homeomorphic to 3-sphere $S^3$, admits a $\Sigma_4$-action. $$H=\langle (1234),(12)(34)\rangle$$ is one of the Sylow 2-subgroups of $\Sigma_4$. It is also isomorphic to $D_8$, the dihedral group of order 8. For $G=\Sigma_4$ or $D_8$, the Bredon cohomology with coefficient $\underline{\Q}$ is given by
\[
H^n_{G}(Z(K;(D^1, S^0)),\underline{\Q})=\left\{
\begin{array}{ll}
\Q^2 & \text{for } n=0;\\
0 & \text{for } n=1,2;\\
\Q & \text{for } n=3.
\end{array}\right.
\]
\end{customthm}

\subsubsection*{\bf Acknowledgements}
The author thanks Doug Ravenel and Mingcong Zeng for inspiring my interest in equivariant homotopy theory; and Fred Cohen for many delightful and helpful conversations over the years.
\section{Group actions on polyhedral products}

\begin{definition}

\begin{enumerate}[label=\roman*)]
\item 
Let $K$ denote an abstract simplicial complex with $m+1$ vertices labeled by 	the set $[m]=\{0,1,2,3,\cdots, m\}$. Thus, $K$ is a subset of the power set $\power([m])$ of $[m]$ such that for each element $\sigma\in K$, we have $\power(\sigma)\subset K$. 
\item
In particular, let $\Delta^m=\power([m])$ denote the abstract simplicial complex given by the full power set. We call it an $m$-simplex.
\item
For any element $\sigma\in K$, elements in $\sigma$ are called vertices of $\sigma$. 
\item
$\sigma\in K$ is called a $k$-simplex of $K$ if $\sigma$ contains $k+1$ vertices, i.e., $|\sigma|=k+1$.
\item
Given any subset $I\subset [m]$, let $K_I$ be the full subcomplex of $K$ consisting of all simplices of $K$ which have all of their vertices in $I$, that is, $K_I=\{\sigma\cap I: \sigma\in K\}$. 
\end{enumerate}

\end{definition}
Let $\mathcal{K}$ be the category of abstract simplicial complexes whose morphisms are monomorphisms of simplicial complexes. and $\mathcal{CW}_*$ are the category of based CW-complexes.
\begin{definition}
\label{Polyhedral product}
Given any based CW-pair $(X,A)$, The polyhedral product functor of $(X,A)$ 
$$Z(-; (X, A)):\mathcal{K}\rightarrow \mathcal{CW}_*$$ is defined as follows:
\begin{enumerate}[label=\roman*)]
\item
Given an abstract simplical complex $K$ with vertex set $[m]$, for any $\sigma\in K$, let $$Z(\sigma; (X, A))=\prod_{i=0}^m Y_i, \text{ where } Y_i=\left\{
\begin{array}{ccl}
X, &\text{if}& i\in \sigma;\\
A, &\text{if} &i\in [m]-\sigma.
\end{array}\right.$$
with $Z(\emptyset; (X, A))=A^{m+1}$.

\item
The polyhedral product with respect to $K$ is 
$$Z(K; (X, A))=\bigcup_{\sigma\in K}Z(\sigma; (X, A))=\colim_{\sigma\in K}Z(\sigma; (X, A)).$$
\end{enumerate}
\end{definition}

\begin{example}
Let $K=\Delta^m$, then $$Z(K; (X, A))=X^{m+1}$$ is the cartesian product of $X$ and $$Z(\partial K; (D^1, S^0))=S^{m},\ \  Z(\partial K; (D^2, S^1))=S^{2m+1}.$$ are both spheres.

More generally, for a manifold $X$ with boundary $A$. The polyhedral product is $Z(\partial K; (X, A))=\partial X^{m+1}$.
\end{example}

\begin{example}{\cite{Coxeter38}}
Let $K_n$ be the boundary of $n$-gon for $n>3$, then $Z(K_n; (D^1, S^0))$ is a Riemann surface of genus $1+(n-4)2^{n-3}$.
\end{example}

Details of the above two examples could be found in \cite{BuchstaberPanov}. The combinatorics property of polyhedral product is closely related to the underlying simplicial complex. For instance, if there is a group $G$ acting on the underlying abstract simplicial complex $K$, then it naturally induces a $G$-action on $Z(K; (X,A))$.

\begin{definition}
Let $K$ be an abstract simplicial complex with vertex set $[m-1]$. Let $G$ be a subgroup of $\Sigma_m$, the symmetric group of $m$ letters. $G$ acts on $K$ if: 

\begin{enumerate}[label=\roman*)]
\item
$G$, as a subgroup of $\Sigma_m$, permutes the vertex set $[m-1]$;
\item
The action on vertices induces the action on the power set $\power([m-1])$;
\item
For any $\sigma\in K$ and any $g\in G$, we have $g\cdot\sigma\in K$, in other words, $K$ is also a $G$-set.
\end{enumerate}
The maximal subgroup of $\Sigma_m$ that acts on $K$ is denoted by $\aut(K)$. 

\end{definition}

\begin{example}
Here are some examples of group actions on simplicial complexes.
\begin{enumerate}
\item
If $K=\Delta^m$ or $\partial \Delta^{m}$, then $\aut(K)=\Sigma_{m+1}$.
\item
If $K_n$ is the boundary of $n$-gon, then $\aut(K)=D_{2n}$, the dihedral group of order $2n$.
\item
If $K$ is the Dynkin diagram $E_7$ or $E_8$, then $\aut(K)=\{e\}$, is the trivial group.
\end{enumerate}
\end{example}

For a topological group $G$, a $G$-space is topological space with a $G$-action and a $G$-map is a continuous map between two $G$-spaces that is equivariant under the $G$-actions.

\begin{definition}
A $G$-CW-complex $X$ with a fixed basepoint is the union of $G$-spaces $X^n$ such that $X^0$ is a disjoint union of orbits $G/H$ that contains the base point, and $X^{n+1}$ is obtained from $X^n$ by attaching $G$-cells $G/H\times D^{n+1}$ along attaching $G$-maps $G/H \times S^n\rightarrow X^n$.
\end{definition}

For any abstract simplicial complex $K$ with vertex set $[m]$, and a subgroup $G <\aut(K)$, $G$ is also a subgroup of $\Sigma_{m+1}$. The polyhedral product $Z(K;(X,A))$ is a subspace of $X^{m+1}$. $G$ acts on $X^{m+1}$ by permuting coordinates. And since $G$ acts on the $K$. We could restrict the $G$-action on $X^{m+1}$ to a well-defined $G$-action on $Z(K;(X,A))$. Moreover, if $\mathcal{K}^G$ is the category of abstract simplicial complexes with $G$-actions, and $\mathcal{CW}_*^G$ is the category of based $G$-CW-complexes and based $G$-maps, then for any given based $CW$-pair $(X,A)$, $Z(-:(X, A))$ is a functor:
$$Z(-:(X, A)):\mathcal{K}^G\rightarrow \mathcal{CW}_*^G$$

The $G$-action on $Z(K; (X,A))$ is not a free $G$-action. Let $\Delta(A)$ be the image of $A$ under the diagonal map 
$$\Delta: X\rightarrow X^{m+1},$$
and $Z(K; (X,A))^G$ be the fixed point set under $G$-action. We have the following lemma:

\begin{lemma}
Let $K$ be an abstract simplicial complex with vertex set $[m]$ and $(X, A)$ a based CW-pair and $G< \aut(K)$. The subgroup $G$ acts on $Z(K; (X,A))$. Then $$\Delta(A)\subset Z(K; (X,A))^G.$$
\end{lemma}

\begin{proof}
The diagonal $\Delta(A)$ is always fixed under permuting coordinates. 
\end{proof}

The group action on polyhedral product $Z(K; (X,A))$ we just defined is not the only group action on it. For example, if $(X, A)$ is a $H$-CW-pair for a group $H$, then $Z(K; (X,A))$ admits  an $H\wr G$-action where $H\wr G$ is the wreath product of the group $H$ with the group $G$. We will discuss other group actions on polyhedral products in later papers.

In the computation of the Bredon cohomology of $Z(K;(X,A))$ with $G$-actions, one crucial information is fixed point set $Z(K; (X,A))^H$ for every subgroup $H<G$. In general, it might be too difficult to give a complete description of those fixed point sets for an arbitrary $G$-space. However, for polyhedral product, we have such a description. 

\begin{definition}
\label{strong quotient}
Let $K$ be an abstract simplicial complex and $G<\aut(K)$ with vertex set $[m]$. The strong quotient of K with respect to $G$, denoted by $K//G$, is another simplicial complex defined as follows:

\begin{enumerate}[label=\roman*)]
\item
Let $O_G$ be the set of orbits of vertices of $K$ under the group action $G$. And denote the number of vertex obrits that do not contained in $K$ by $k_G=|O_G-K|$. 
\item
Let $V_G$ be the set of orbits contained in $K$, i.e. $V_G=O_G\cap K$.
\item
The map $q: [m]\rightarrow V$ is the natural projection.
\item
Define $\widehat{q}:K\rightarrow \power(O_G)$ by following:

For any $\sigma\in K$,
\begin{enumerate}[label=\alph*)]
\item
If there exists some $v\in O_G$ such that  $v\cap \sigma\not=\emptyset$ and $v\not\subset \sigma$, then $\widehat{q}(\sigma)=\emptyset$;
\item
Otherwise, $\sigma$ is a union of several orbits, $\widehat{q}(\sigma)=\{v\in O_G: v\subset o\}$ is the set of all orbits contained in $\sigma$.
\end{enumerate}

\item
Let $S$ be the set of orbits of $K$ under $G$-action. An element in $S$ is donoted by $\overline{\sigma}$ where $\sigma\in K$ is the orbit representative.
\item
The image $\widehat{q}(\overline{\sigma}):= \widehat{q}(\sigma)$ is a subset of $V_G$. It is well-defined since $G$ acts on $\aut (K)$.
\item
Define $K//G:=\{\widehat{q}(\overline{\sigma}):\overline{\sigma}\in S, \text{and } \widehat{q}(\overline{\sigma})\subset V_G\}$ is the abstract simplicial complex with vertex set $V_G$.
\end{enumerate}
\end{definition}

\begin{example}
\label{simplex quotient}
Let $K=\Delta^m$, then $\aut(K)=\Sigma_{m+1}$. For any $G<\aut(K)$, $$K//G=\Delta^k\text{, for some } k\leq m.$$
\end{example}

The formula for the fixed point set using the strong quotient is given next.
\begin{theorem}
\label{fixed point theorem}
Let $K$ be an abstract simplicial complex with vertex set $[m]$ and $(X, A)$ a based CW-pair. For any subgroup $G< \aut(K)$, $G$ acts on $Z(K; (X,A))$. Then 
$$Z(K; (X, A))^G=Z(K//G; (X,A))\times A^{k_G}.$$
\end{theorem}

\begin{proof}
For any vertex $k\in[m]$, consider the orbits $G\cdot k$. For any $(x_0, \cdots, x_m)\in Z(K; (X, A))^G$, we have $x_i=x_k$ for any $i\in G\cdot k$. Hence the fixed points contain in the intersection of $Z(K; (X,A))$ with the diagonal images corresponding to all the vertex orbits $G\cdot k$ for $k\in [m]$.
\begin{enumerate}
\item
If $G\cdot k\notin K$, then $x_k$ has to be contained in $A$. And the orbit $G\cdot k$ corresponds to one copy of $A$ in the final fixed point set.
\item
If $G\cdot k \in K$, then $x_k$ could be any point in $X$. And the $G\cdot k$ corresponds to an one vertex in $K//G$.
\end{enumerate}
Depend on whether $G\cdot k$ belongs to $K$ or not, we could identify the intersection of diagonal images corresponding to a the vertex orbits $G\cdot k$ by its homeomorphic image $X$ or $A$. On the level of simplicial complex, this identification 
corresponds to collapse the simplicial complex $K$ with respect to the vertex orbits $O_G$ and omitting all the orbits that not contain in $K$, which gives the strong quotient $V//G$. Finally, we have $$Z(K; (X, A))^G=Z(K//G; (X,A))\times A^{k_G}.$$
\end{proof}

\begin{corollary}
Consider the real moment-angle complex $Z(K; (D^1, S^0))$ with $G$-action. Then the fixed point set $Z(K; (D^1, S^0))^G$ is a disjoint union of $2^k$'s copies of a different real moment-angle complex $Z(K//G; (D^1, S^0))$ for some $k$. 
\end{corollary}
We are going to give several examples, starting with the trivial ones.
\begin{example}
As in Example \ref{simplex quotient}, let $K=\Delta^m$, for any $G<\aut(K)$, $$K//G=\Delta^k$$ for some $k\leq m$. Meanwhile $Z(K; (X,A))=X^{m+1}$ and $Z(K//G; (X,A))=X^{k+1}$. In the particular setting, Theorem \ref{fixed point theorem} reduces to the case of fixed point sets of a product space $X^{m+1}$ under the swapping action. 
\end{example}

\begin{figure}[h]
\centering
\begin{tikzpicture}
\draw (0,0) -- (0,2);
\draw (0,0) -- (-1.732,-1);
\draw (0,0) -- (1.732,-1);
\filldraw [black] (0,0) circle (2pt) node[below] {0};
\filldraw [black] (0,2) circle (2pt)node[above right] {1};
\filldraw [black] (-1.732,-1) circle (2pt)node[below left] {2};
\filldraw [black] (1.732,-1) circle (2pt)node[below right] {3};
\end{tikzpicture}
\caption{} 
\label{fig:star}
\end{figure}

\begin{example}
Let $K$ be the simplicial complex given by Figure \ref{fig:star}. Then $$Z(K; (X, A))=(X\times X\times A\times A)\cup (X\times A\times X\times A)\cup (X\times A\times A\times X).$$ 
Let $G=C_3$ be the cyclic group of order 3. The group $G$ acts on $K$ by sending $1\rightarrow 2\rightarrow 3\rightarrow 1$ and fixing $0$. Then $K//G$ is a single vertex. And $$Z(K; (X, A))^G=Z(K//G; (X,A))\times A^{k_G}=X\times A.$$
\end{example}

\begin{figure}[h]
\centering
\begin{tikzpicture}
\filldraw [black] (1,-1) circle (2pt) node[below right] {0};
\filldraw [black] (0,0) circle (2pt) node[above left] {1};
\filldraw [black] (-1.5,-0.5) circle (2pt) node[left] {2};
\filldraw [black] (1,2) circle (2pt) node[above right] {3};
\filldraw [black] (0,3) circle (2pt) node[above] {4};
\filldraw [black] (-1.5,2.5) circle (2pt) node[above left] {5};
\filldraw[color=black, fill=gray, ultra thick] (1,-1) -- (0,0) -- (-1.5,-0.5) -- cycle;
\filldraw[color=black, fill=gray, ultra thick] (1,2) -- (0,3) -- (-1.5,2.5) -- cycle;
\draw[ultra thick] (1,-1) -- (1,2);
\draw[ultra thick] (0,0) -- (0,3);
\draw[ultra thick] (-1.5,-0.5) -- (-1.5,2.5);
\end{tikzpicture}
\caption{} 
\label{fig:trilinder}
\end{figure}

\begin{example}
Let $K$ be the simplicial complex given by Figure \ref{fig:trilinder}. More specifically,
\[ 
\begin{array}{rl}
K =& \{\{0\},\{1\},\{2\},\{3\},\{4\},\{5\},\\
   & \{0,1\},\{1,2\},\{0,2\},\{3,4\},\{4,5\},\{3,5\},\{0,3\},\{1,4\},\{2,5\},\\
  & \{0,1,2\},\{3,4,5\}\}.
\end{array}
\]

Let $G=C_3$ be the cyclic group of order 3. Define the $G$-action on $K$ by sending $0\rightarrow 1\rightarrow 2\rightarrow 0$ and $3\rightarrow 4\rightarrow 5\rightarrow 3$. Then the strong quotient $K//G=\partial \Delta^1$.  And $$Z(K; (X, A))^G=Z(K//G; (X,A))\times A^{k_G}=(X\times A)\cup (A\times X).$$ 
\end{example}

One interesting examples is the symmetric group $\Sigma_{n+1}$ acts on spheres $S^n$. Let $K=\partial \Delta^{n}$ be the boundary of $n$-simplex. Then we have $\aut(K)=\Sigma_{n+1}$. The polyhedral product $Z(K; (D^1,S^0))=S^{n}$ is an $n$-sphere. Let $G$ be any subgroup of $\aut(K)$. Then $G$ acts on $S^{n}$ through the polyhedral product structure. Let $V$ be the set of orbits in the vertex set $[n]$. There are two possibilities:
\begin{enumerate}
\item
There is only one orbit, i.e., $V=\{[n]\}$, then $K//G=\emptyset$ and $Z(K; (D^1,S^0))^G=S^0$.
\item
There are $k$ orbits for some $k> 1$, then each orbits $\sigma\in V$ is a proper subset of $[n]$ that contained in $K$. Hence $K//G=\partial \Delta^{k-1}$. Moreover, $Z(K; (D^1,S^0))^G=Z(K//H; (D^1,S^0))=S^{k-1}$.
\end{enumerate}
\begin{corollary}
\label{sphere fixed point}
Consider the $n$-sphere $S^n$ defined by $Z(K, (D^1, S^0))$ where $K=\partial \Delta^n$.  It admits a $\Sigma_{n+1}$-action. For any subgroup $G<\Sigma_{n+1}$, the fixed point set is a $k$-sphere $S^k$ for some $k\leq n$. In particular, $k=n$ if and only if the subgroup $G$ is the trivial group.\hfill $\square$
\end{corollary}

Let X be a $G$-space and $H$ is a subgroup of $G$. The Weyl group $WH$ for a subgroup $H<G$, is given by $$WH=N_G(H)/H,$$ where $N_G(H)$ is the normalizer of $H$ in $G$. The Weyl group $WH$ acts on the fixed point set $X^H$ naturally. In fact, for any $[g]\in WH$, let $g$ be its representative in $N_G(H)$. Then for any $x\in X^H$ and $h\in H$, $$h.(g.x)=g.(g^{-1}hg.x)=g.x$$ 

For the purpose of computing Bredon cohomology of polyhedral products, we are interested in the homology of the fixed point sets as well as the Weyl group action on the homology. For the examples we have discussed above, especially for the sphere case, since we know that the fixed point set is a lower dimension sphere. We could obtain the homology by counting vertex orbits. But in general, it is a difficult task to compute the homology of certain polyhedral products and apply K\"{u}nneth formula. Bahri, Bendersky, Cohen and Gitler develop several tools in computing cohomology of polyhedral product. In \cite{BBCG15}, the authors construct a filtration of polyhedral product.

\begin{definition}
Let $K$ be a simplicial complex with $m+1$ vertices. $\Delta^m$ is filtered by the left lexicographical ordering of all faces. Then filter $K$ by
$$F_t(K)=K\cap F_t(\Delta^m),$$
and 
$$F_t(Z(K;(X, A)))=Z(F_t(K); (X, A)).$$
\end{definition}

Using this filtration, there are natural spectral sequences to compute homology and cohomology of $Z(K;(X, A))$. 

In another paper \cite{BBCG10}, the authors give the stable decomposition of polyhedral product,

\begin{theorem}[\cite{BBCG10}]
Let $(X, A)$ be connected, pointed CW-pairs, and $K$ is an abstract simplicial complex with vertex set $[m]$. we have a pointed homotopy equivalence $$f: \Sigma(Z(K;(X, A)))\rightarrow \Sigma(\bigvee_{I\subset[m]}\hat{Z}(K_I; (X, A)))$$
where $\hat{Z}$ is the polyhedral smash product and $K_I$ is the largest subcomplex of $K$ that contains $I$.
\end{theorem}

Later in \cite{Ali14}, Al-Raisi shows that the stable decomposition of a polyhedral product is also equivariant under the group action we defined in the beginning of this section. Therefore stable decomposition provide a general method to compute the Weyl group action on the homology of the fixed point set.

For example, the inclusion $$Z(K;(X,*))\subset X^m$$ induces an $\Aut(K)$-equivariant epimorphism $$H^*(X^m)\rightarrow H^*(Z(K;(X,*))).$$

\section{Bredon coefficient systems}
\begin{definition}
The canonical orbit category of group $G$, denoted by $\og$, is the category whose objects are G-spaces $G/H$ and morphisms are $G$-maps. 
\end{definition}

There is a $G$-map $f:G/H \rightarrow G/K$ if and only if $gHg^{-1}<K$. Notice that if $f(eH)=gK$ for some $G\in G$, then $$eK=f(g^{-1}H)=f(g^{-1}hg\cdot g^{-1}H)= g^{-1}hg\cdot f(g^{-1}H)=g^{-1}hg \cdot eK$$ for any $h\in H$. Hence $g^{-1}Hg<K$.

\begin{definition}
Let $R$ be a commutative ring and $Mod_R$ be the category of $R$-modules. A Bredon coefficient system over $R$ is a contravariant functor $\og\rightarrow Mod_R$. The category of Bredon coefficient systems over $R$ is denoted by $\cg^R$. When $R=\mathbb{Z}$, we  use $\cg$ for simplicity.
\end{definition}

\begin{example}
Given a based $G$-CW-complex, the n-th equvariant homotopy group $\underline{\pi_n}(X)$ for $n\geq 2$, is a Bredon coefficient system given by $$\underline{\pi_n}(X)(G/H)=\pi_n(X^H)$$
\end{example}

\begin{example}
Let X be a $G$-CW-complex, we can define the cellular chain complex of coefficient systems $\underline{C_*}(X)$, where
$$\underline{C_n}(X)(G/H)=H_n((X^n)^H, (X^{n-1})^H, R).$$
And define $$\underline{H_n}(X)=H_n(\underline{C_*}(X)).$$
\end{example}

By a general categorical argument, the category $\cg^R$ is an Abelian category with enough injectives. So we could talk about homological algebra concept such as homology and cohomology of chains and cochains in this category.

\begin{definition}
\label{Bredon chomology and homology}
Let $\underline{M}$ be a Bredon coefficient system and $X$ be a $G$-CW-complex. Define a cochain complex
$$\underline{C^n}(X; \underline{M})=\Hom_{\cg^R}(\underline{C_n}(X),\underline{M}),$$
Its cohomology, $$H^*_G(X, \underline{M}):=H^*(\underline{C^*}(X;\underline{M}))$$ is called the Bredon cohomology of $X$ with coefficient $\underline{M}$. To define Bredon homology, we need a covariant functor $\underline{N}:\og\rightarrow Mod_R$ as a coefficient system. Define the cellular chain by 
$$C_n(X; \underline{N})=\underline{C_n}(X)\otimes_{\cg^R} \underline{N}=\int^{G/H}\underline{C_n}(X)(G/H)\otimes_R \underline{N}(G/H)$$
In other word, the tensor product $\otimes_{\cg^R}$ is given by the coend of the two functors. Bredon homology of $X$ with coefficient $N$ is given by 
$$H^G_n(X,\underline{N})=H_n(C_*(X; \underline{N}))$$
\end{definition}

\begin{theorem}[\cite{Bredon},\cite{May96}]
\label{universal coefficient spectral sequence}
There are universal coefficient spectral sequences
$$E_2^{p,q}=\Ext_{\cg^R}^{p,q}(\underline{H_*}(X), \underline{M})\Rightarrow H_G^n(X, \underline{M}),$$
and
$$E^2_{p,q}=\Tor^{\cg^R}_{p,q}(\underline{H_*}(X), \underline{N})\Rightarrow H^G_n(X, \underline{M}).$$
\end{theorem}

\subsection{Bredon coefficient systems in the language of path algebras}
Using the language of path algebra, we could also give a more ``algebraic'' description of Bredon cohomology and homology. In this subsection, We will define path algebra and give an alternative defnition of Bredon coefficient systems as well as Bredon homology and cohomology. The readers who are interested in this topic could check for more details of path algebras and the related representation theory in the summary book \cite{ASS06}.

\begin{definition}
A quiver $Q=(V, E, s, t)$ is a directed graph, i.e. a graph that associates each edge a direction, where $V$ is the set of vertices and $E$ is the set of edges along with two maps $s,t: E\rightarrow V$ that for each edge $\alpha\in E$, the images $s(\alpha)$ and $t(\alpha)$ are the source and target of the edge $\alpha$ respectively. 
\end{definition}

\begin{example}
A small category $\mathscr{C}$ is naturally a quiver whose objects and morphisms are vertices and edges of the quiver respectively. We denote the quiver associated  with $\mathscr{C}$ by $Q_\mathscr{C}$. 
\end{example}

\begin{definition}
For any quiver $Q=(V, E, s, t)$, the set of paths of $Q$, denoted by  $P_Q$, consists of the following elements:
\begin{enumerate}[label=\roman*)]
\item
For each vertex $v\in V$, there is a trivial path $e_v$, and set $s(e_v)=t(e_v)=v$;
\item
All the finite sequences $\alpha_n\alpha_{n-1}\cdots\alpha_1$ where $\alpha_i\in E$ for each $i$ and $t(\alpha_k)=s(\alpha_{k+1})$ for $k=1,2,\cdots, n-1$. In other word, an actual path on the quiver.
\end{enumerate}
\end{definition}

Moreover, we can define a multiplication $\circ$ on $P_Q$. For any two paths $\alpha_n\alpha_{n-1}\cdots\alpha_1$, $\beta_m\beta_{m-1}\cdots\beta_1 $ in $P_Q$,
 
\[
\alpha_n\cdots\alpha_1 \circ \beta_m\cdots\beta_1 =\left\{
\begin{array}{cl}
\alpha_n\cdots\alpha_1\beta_m\cdots\beta_1 & \text{if } t(\beta_m)=s(\alpha_1),\\
0 & \text{otherwise.}
\end{array}
\right.
\]

\begin{definition}
Let $R$ be a commutative ring, the path algebra $RQ$ is an associative algebra with basis $P_Q$ whose multiplication is linearly induced by the multiplication $\circ$ on $P_Q$ over $R$. In addition, if $V$ is an finite set, then the algebra $RQ$ has a multiplicative identity $$1_{RQ}=\sum_{v\in V}e_v,$$ and these $e_v$'s are idempotents of $RQ$.
\end{definition}
\begin{example}
Let $Q$ be the quiver given by Figure \ref{fig:An}. Then over a field $k$, the path algebra $kQ$ is a $k$-algebra of dimension $\frac{n^2+n}{2}$. And it is isomorphic to the algebra of $n\times n$ upper triangular matrices over $k$.

\begin{figure}[h]
\centering
\begin{tikzpicture}
\filldraw [black] (-5,0) circle (2pt) node[above] {1};
\filldraw [black] (-3,0) circle (2pt) node[above] {2};
\filldraw [black] (-1,0) circle (2pt) node[above] {3};
\filldraw [black] (4,0) circle (2pt) node[above] {n-1};
\filldraw [black] (6,0) circle (2pt) node[above] {n};
\filldraw [black] (1.4,0) circle (1pt);
\filldraw [black] (1.8,0) circle (1pt);
\filldraw [black] (2.2,0) circle (1pt);
\draw[ultra thick, ->](-4.8,0)--(-3.2, 0);
\draw[ultra thick, ->](-2.8,0)--(-1.2, 0);
\draw[ultra thick, ->](-0.8,0)--(0.8, 0);
\draw[ultra thick, ->](2.8,0)--(3.8, 0);
\draw[ultra thick, ->](4.2,0)--(5.8, 0);
\end{tikzpicture}
\caption{} 
\label{fig:An}
\end{figure}
\end{example}

\begin{definition}
A relation of a quiver $Q$ is a subspace of RQ spanned by linear combinations of paths having a common source and a common target. Let $S$ be a set of relations  of $Q$. The path algebra with relation $S$, denoted by $RQ_S$ is $RQ/I_S$ where $I_S$ is the two-sided ideal generated by $S$. 
\end{definition}

For an algebra $A$ over $R$, Let $Mod_{A}$ be the category of right-$A$-modules and ${}_{A}Mod$ be the category of left-$A$-modules.

\begin{lemma}
If $G$ is a finite group, the canonical orbit category $\og$ is a small category with finitely many objects. Let $Q= Q_{\og}$ be the quiver associated with $\og$. And $S$ is the set of relations given by the equivalences of morphisms in the category $\og$. Then a Bredon coefficient system is naturally a right-$RQ_S$-module. And there is an equivalence of Abelian categories $\cg^R\rightarrow Mod_{RQ_S}$. A covariant functor $\underline{N}:\og\rightarrow Mod_R$ is naturally a left-$RQ_S$-module and the coend of two such functors is the tensor product of a right $RQ_S$-module with a left $RQ_S$-module over $RQ_S$.
\end{lemma}

\begin{example}
Let $G=C_2$, the cyclic group of order 2.  the canonical orbit category is shown in Figure \ref{fig:C2}. The path algebra with relations corresponds to $\og$ has dimension 5, and a basis is given by $\{e_{C_2/C_2}, e_{C_2/0}, \alpha, \beta, \beta\alpha\}$. Notice that the loop $\beta$ is given by the Weyl group action. And since the Weyl group for trivial group is the whole group $C_2$, then set of relations in this case is $\{\beta^2\}$.

\begin{figure}[h]
\centering
\begin{tikzpicture}[scale=0.7]
\filldraw[black] (0,0) circle (2pt) node[above left] {$C_2/0$};
\filldraw[black] (0,3) circle (2pt) node[above] {$C_2/C_2$};
\node[left] at (0,1.5) {$\alpha$};
\node at (0,-1.5) {$\beta$};
\draw[ultra thick, ->](0,0.2)--(0,2.8);
 \draw [ultra thick, ->] (-0.3,0) arc(105:435:1);
\end{tikzpicture}
\caption{Canonical orbit category of $C_2$} 
\label{fig:C2}
\end{figure}
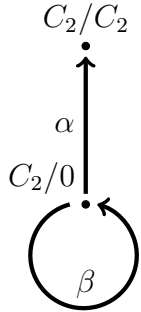
\end{example}

\subsection{Reduced Bredon coefficient systems}
In order to compute Bredon cohomology, we need to reduce the Bredon coefficient systems as simply as possible. Since the conjugation of subgroups in $G$ induces isomorphisms in the canonical orbit category $\og$, it is enough that we just consider one subgroup of $G$ for each conjugate class to obtain a category which is equivalent to $\og$ but with fewer objects. In explicit, Let $\widehat{\og}$ be the fully subcategory of $\og$ whose object set consists one and only one object for each isomorphism class in $\og$. This definition is not canonical, but $widehat{\og}$ is unique up to category equivalence.

We define a reduced Bredon coefficient system over a commutative ring $R$ to be a contravariant functor $$\underline{M}: \widehat{\og}\rightarrow Mod_R.$$

The following lemma follows immediately from the definition.
 
\begin{lemma}
The category of reduced Bredon coefficient systems over $R$ is equivalent to $\cg^R$.
\end{lemma}

Therefore in the future discussion, we will only talk about reduced Bredon coefficient system and reduced canonical orbit category. In order to simplify the notation, denote the reduced canonical orbit category again by $\og$.

\begin{example}
Let $G=\Sigma_3$, the (reduced) canonical orbit category is shown in Figure \ref{fig:S3}. Notice that $\sigma,\tau$ are the generators of the Weyl group $W\{e\}=\Sigma_3$ and $\iota$ is the generator of $W\langle(123)\rangle=\Z_2$. The set of relations is $\{\beta\alpha-\delta\gamma, \iota^2, \sigma^3, \tau^2, \tau\sigma\tau-\sigma^2,\sigma\tau\sigma-\tau\}$.

\begin{figure}[h]
\centering
\begin{tikzpicture}[scale=0.7]
\filldraw[black] (0,0) circle (2pt);
\node[above] at (0, 0.3){$\Sigma_3/\{e\}$};
\filldraw[black] (-3,3) circle (2pt) node[right] {$\Sigma_3/\langle (123)\rangle$};
\filldraw[black] (3,4) circle (2pt) node[right] {$\Sigma_3/\langle (12)\rangle$};
\filldraw[black] (0,6) circle (2pt) node[above] {$\Sigma_3/\Sigma_3$};
\node[below left] at (-1.5,1.5) {$\alpha$};
\node[above left] at (-1.5,4.5) {$\beta$};
\node[below right] at (1.5,2) {$\gamma$};
\node[above right] at (1.5,5) {$\delta$};
\node at (0,-1.5) {$\sigma$};
\node at (0,-2.5) {$\tau$};
\node at (-4.5,3) {$\iota$};
\draw[ultra thick, ->](-0.2,0.2)--(-2.8, 2.8);
\draw[ultra thick, ->](0.2,0.2)--(2.8, 3.8);
\draw[ultra thick, ->](-2.8,3.2)--(-0.2, 5.8);
\draw[ultra thick, ->](2.8,4.2)--(0.2, 5.8);
\draw [ultra thick, ->] (-0.3,0) arc(105:435:1);
\draw [ultra thick, ->] (-3,3.3) arc(15:345:1);
\draw [ultra thick, ->] (-0.45,0) arc(110:430:1.5);
\end{tikzpicture}
\caption{} 
\label{fig:S3}
\end{figure}
\end{example}

\section{Bredon cohomology over $\mathbb{Q}$}

In order to apply the universal coefficient spectral sequence, we need to find an injective resolution for a given coefficient system. However, it could be a difficult task. It is not known to the author whether there is a general method to construct an injective resolution for any coefficient system over integers. However, over $\Q$, Doman in \cite{Doman88} constructed the injective envelope for any Bredon cefficient system over rationals.

\begin{definition}
For any subgroup $H<G$, let $V_H$ be a left $\Q(WH)$-module, define a Bredon coefficient system $I(V_H)$ by 
$$\underline{I(V_H)}(G/K)= \Hom_{\Q(WH)}(\Q((G/K)^H), V_H).$$ 
For a $G$-map $f:G/K\rightarrow G/K'$, the map $\overline{f}: \Q((G/K)^H)\rightarrow \Q((G/K')^H)$ is induced by $f$. Then 
$$\underline{I(V_H)}(f): \underline{I(V_H)}(G/K')\rightarrow \underline{I(V_H)}(G/K)$$ is defined by $\underline{I(V_H)}(f)(g)=g\circ \overline{f}$ where $g\in I(V_H)(G/K')$. 
\end{definition}

\begin{lemma}[\cite{Doman88}]
\label{injecitveness}
The coefficient system $I(V_H)$ is an injective object in the category $\cg^{\Q}$.
\end{lemma}

Moreover, the injective coefficient system $\underline{I(V_H)}$ could be used to construct an injective envelope for any Bredon coefficient $\underline{M}$ over $\Q$. Let $V_{\{e\}}=\underline{M}(G/\{e\})$. And in general, let $$V_H=\bigcap_{K<H}\Ker\underline{M}(f_{K,H})$$ where $f_{K,H}:G/K\rightarrow G/H$ is the projection in the canonical orbit category.

\begin{theorem}[\cite{Doman88}]
\label{injective envelope}
For any coefficient system $\underline{M}$ over $\Q$, the map $f:\underline{M}\rightarrow\oplus \underline{I(V_H)}$ is an injective envelope of $\underline{M}$, where the direct sum is over all conjugate classes of $G$.
\end{theorem}

\begin{corollary}
The global (injective) dimension of $\cg^\Q$ is less than $L-1$ where $L$ is the largest length of subgroup chains in $G$.
\end{corollary}

The fixed point set theorem which we obtained in section 1 give us the access to $\underline{H_*}((Z(K;(X,A))))$, the coefficient system defined by the homology of fixed point sets. Meanwhile Doman's result provide a general method to find a injective resolution for a rational Bredon coefficient system.  Combine those, we have the both parts of the inputs of the universal coefficient spectral sequence for Bredon cohomology, at least for some examples. Let $K=\partial \Delta_3$, the the automorphism group of $K$ is $\Aut(K)=\Sigma_4$. The polyhedral product $Z(K;(D^1, S^0))$ is homeomorphic $S^3$ with a $\Sigma_4$-action. We will compute the Bredon cohomology of $Z(K;(D^1, S^0))$ for groups $\Sigma_4$ and $D_8$(as the Sylow 2-subgroup of $\Sigma_4$).

Let $\underline{\Q}$ be the constant $\Q$ coefficient system.

\begin{lemma}
The constant coefficient system $\underline{\Q}$ is injective.
\end{lemma}

\begin{proof}
Take $H=\{e\}$, it follows that $WH=G$ and $(G/K)^H=G/K$. Then for trivial $WH$-module $V_H=\Q$, we have 
$$\Hom_{\Q(WH)}(\Q(G/K),V_H)=\Q.$$
Hence $\underline{I(V_H)}=\underline{\Q}$. 
\end{proof}

\begin{corollary}
The universal coefficient spectral sequence $$E_2^{p,q}=\Ext_{\cg^R}^{p,q}(\underline{H_*}(X), \underline{\Q})\Rightarrow H_G^n(X, \underline{\Q}),$$ collapses at $E_2$-page. Therefore $$H^n_{G}(X,\underline{\Q})= \Hom_{\cg^{\Q}}(\underline{H_n}(X), \underline{\Q}).$$
\end{corollary}

\subsection{Computation for $D_8$}

The standard presentation of the dihedral group of order 8 is given by 
$$G=D_8=\langle\sigma,\tau:\sigma^4=\tau^2=e, \tau\sigma\tau=\sigma^{-1}\rangle.$$
The (reduced) canonical orbit category is shown in the left diagram of Figure \ref{fig:D8}. Notice that the Weyl group actions are not drawn in the figure. $W\{e\}=G$, $W\langle \sigma^2\rangle=\Z_2\times Z_2$, $WG=0$ and all other subgroups has Weyl group $\Z_2$. For the convenience of later reference we name the objects in $\og$ as in the right diagram of Figure \ref{fig:D8}. 

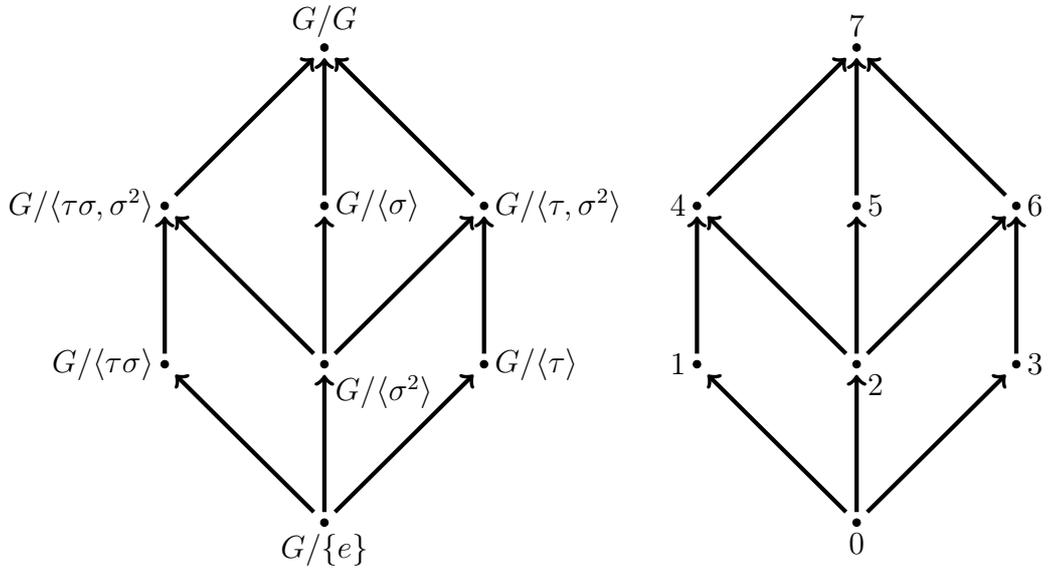
\begin{figure}[h]
\centering
\begin{tikzpicture}[scale=0.7]
\filldraw[black] (0,0) circle (2pt) node[below] {$G/\{e\}$};
\filldraw[black] (0,3) circle (2pt) node[below right] {$G/\langle\sigma^2\rangle$};
\filldraw[black] (-3,3) circle (2pt) node[left] {$G/\langle\tau\sigma\rangle$};
\filldraw[black] (3,3) circle (2pt) node[right] {$G/\langle\tau\rangle$};
\filldraw[black] (0,6) circle (2pt) node[right] {$G/\langle\sigma\rangle$};
\filldraw[black] (-3,6) circle (2pt) node[left] {$G/\langle\tau\sigma,\sigma^2\rangle$};
\filldraw[black] (3,6) circle (2pt) node[right] {$G/\langle\tau, \sigma^2\rangle$};
\filldraw[black] (0,9) circle (2pt) node[above] {$G/G$};
\draw[ultra thick, ->](0,0.2)--(0,2.8);
\draw[ultra thick, ->](-0.2,0.2)--(-2.8,2.8);
\draw[ultra thick, ->](0.2,0.2)--(2.8,2.8);
\draw[ultra thick, ->](0,3.2)--(0,5.8);
\draw[ultra thick, ->](-3,3.2)--(-3,5.8);
\draw[ultra thick, ->](3,3.2)--(3,5.8);
\draw[ultra thick, ->](-0.2, 3.2)--(-2.8, 5.8);
\draw[ultra thick, ->](0.2,3.2)--(2.8,5.8);
\draw[ultra thick, ->](-2.8,6.2)--(-0.2,8.8);
\draw[ultra thick, ->](2.8,6.2)--(0.2,8.8);
\draw[ultra thick, ->](0,6.2)--(0,8.8);
\filldraw[black] (10,0) circle (2pt) node[below] {0};
\filldraw[black] (10,3) circle (2pt) node[below right] {2};
\filldraw[black] (7,3) circle (2pt) node[left] {1};
\filldraw[black] (13,3) circle (2pt) node[right] {3};
\filldraw[black] (10,6) circle (2pt) node[right] {5};
\filldraw[black] (7,6) circle (2pt) node[left] {4};
\filldraw[black] (13,6) circle (2pt) node[right] {6};
\filldraw[black] (10,9) circle (2pt) node[above] {7};
\draw[ultra thick, ->](10,0.2)--(10,2.8);
\draw[ultra thick, ->](9.8,0.2)--(7.2,2.8);
\draw[ultra thick, ->](10.2,0.2)--(12.8,2.8);
\draw[ultra thick, ->](10,3.2)--(10,5.8);
\draw[ultra thick, ->](7,3.2)--(7,5.8);
\draw[ultra thick, ->](13,3.2)--(13,5.8);
\draw[ultra thick, ->](9.8, 3.2)--(7.2, 5.8);
\draw[ultra thick, ->](10.2,3.2)--(12.8,5.8);
\draw[ultra thick, ->](7.2,6.2)--(9.8,8.8);
\draw[ultra thick, ->](12.8,6.2)--(10.2,8.8);
\draw[ultra thick, ->](10,6.2)--(10,8.8);
\end{tikzpicture}
\caption{$\mathscr{O}_{D8}$} 
\label{fig:D8}
\end{figure}

\begin{figure}[h]
\centering
\begin{tikzpicture}[scale=0.7]
\filldraw[black] (0,0) circle (2pt) node[below] {$S^3$};
\filldraw[black] (0,3) circle (2pt) node[below right] {$S^1$};
\filldraw[black] (-3,3) circle (2pt) node[left] {$S^2$};
\filldraw[black] (3,3) circle (2pt) node[right] {$S^1$};
\filldraw[black] (0,6) circle (2pt) node[right] {$S^0$};
\filldraw[black] (-3,6) circle (2pt) node[left] {$S^1$};
\filldraw[black] (3,6) circle (2pt) node[right] {$S^0$};
\filldraw[black] (0,9) circle (2pt) node[above] {$S^0$};
\draw[ultra thick, <-](0,0.2)--(0,2.8);
\draw[ultra thick, <-](-0.2,0.2)--(-2.8,2.8);
\draw[ultra thick, <-](0.2,0.2)--(2.8,2.8);
\draw[ultra thick, <-](0,3.2)--(0,5.8);
\draw[ultra thick, <-](-3,3.2)--(-3,5.8);
\draw[ultra thick, <-](3,3.2)--(3,5.8);
\draw[ultra thick, <-](-0.2, 3.2)--(-2.8, 5.8);
\draw[ultra thick, <-](0.2,3.2)--(2.8,5.8);
\draw[ultra thick, <-](-2.8,6.2)--(-0.2,8.8);
\draw[ultra thick, <-](2.8,6.2)--(0.2,8.8);
\draw[ultra thick, <-](0,6.2)--(0,8.8);
\end{tikzpicture}
\caption{Fixed point set of $S^3$ with $D_8$-action} 
\label{fig:S3D8}
\end{figure}
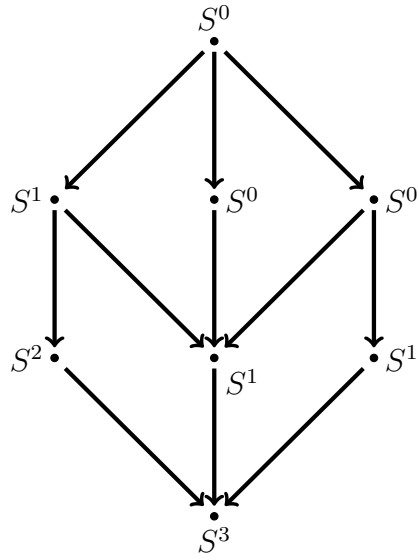
By Corollary \ref{sphere fixed point} and the argument before it, the fixed points of $S^3$ for any subgoup $H<D_8$ is given by Figure \ref{fig:S3D8}. We need to compute  the $\Q$-dimension of the homomorphism groups between coefficient systems in the category $\cg^{\Q}$. Notice that a homomorphism between two coefficient systems $\underline{M}$ and $\underline{N}$ is a collection of $\Q$-linear maps $$f_i: \underline{M}(i)\rightarrow \underline{N}(i)$$
that makes the full diagram commute, like in Figure \ref{D_8H_0Q}. The following  is a trivial but handy lemma in the computation the homomorphism groups.

\begin{lemma}
\label{handy}
Let $M$ and $N$ be two $Q$-spaces, if we have the following commutative diagram 
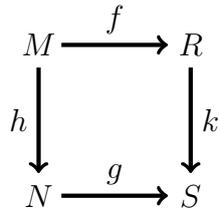
\begin{figure}[H]
\centering
\begin{tikzpicture}[scale=1]
\node[] at (0,0) {$N$};
\node[] at (0,2) {$M$};
\node[] at (2,0) {$S$};
\node[] at (2,2) {$R$};
\node[above] at (1,0) {$g$};
\node[above] at (1,2) {$f$};
\node[left] at (0,1) {$h$};
\node[right] at (2,1) {$k$};
\draw[ultra thick, ->](0,1.7)--(0,0.3);
\draw[ultra thick, ->](2,1.7)--(2,0.3);
\draw[ultra thick, ->](0.3,0)--(1.7,0);
\draw[ultra thick, ->](0.3,2)--(1.7,2);
\end{tikzpicture}
\end{figure}
Then we have 
\begin{enumerate}[label=\roman*)]
\item
The map $f=0$, if $g=0$ and $k$ is an isomorphsim.
\item
The map $f$ is uniquely determined by $g$ if $h$ and $k$ are isomorphisms.
\item
The map $f$ is uniquely determined by $h$ if $g$ and $k$ are isomorphisms. 
\end{enumerate}
\hfill $\square$
\end{lemma}
Firstly, let us compute $$H^0_{D_8}(Z(K;(D^1, S^0)),\underline{\Q})= \Hom_{\cg^{\Q}}(\underline{H_0}(Z(K;(D^1, S^0)), \underline{\Q}).$$ The commutative diagram is shown in Figure \ref{D_8H_0Q}.
\begin{figure}[h]
\centering
\begin{tikzpicture}[scale=0.7]
\filldraw[black] (0,0) circle (2pt) node[below] {$\Q$};
\filldraw[black] (0,3) circle (2pt) node[below right] {$\Q$};
\filldraw[black] (-3,3) circle (2pt) node[left] {$\Q$};
\filldraw[black] (3,3) circle (2pt) node[right] {$\Q$};
\filldraw[black] (0,6) circle (2pt) node[right] {$\Q^2$};
\filldraw[black] (-3,6) circle (2pt) node[left] {$\Q$};
\filldraw[black] (3,6) circle (2pt) node[right] {$\Q^2$};
\filldraw[black] (0,9) circle (2pt) node[above] {$\Q^2$};
\draw[ultra thick, <-](0,0.2)--(0,2.8);
\draw[ultra thick, <-](-0.2,0.2)--(-2.8,2.8);
\draw[ultra thick, <-](0.2,0.2)--(2.8,2.8);
\draw[ultra thick, <-](0,3.2)--(0,5.8);
\draw[ultra thick, <-](-3,3.2)--(-3,5.8);
\draw[ultra thick, <-](3,3.2)--(3,5.8);
\draw[ultra thick, <-](-0.2, 3.2)--(-2.8, 5.8);
\draw[ultra thick, <-](0.2,3.2)--(2.8,5.8);
\draw[ultra thick, <-](-2.8,6.2)--(-0.2,8.8);
\draw[ultra thick, <-](2.8,6.2)--(0.2,8.8);
\draw[ultra thick, <-](0,6.2)--(0,8.8);
\filldraw[black] (10,0) circle (2pt) node[below] {$\Q$};
\filldraw[black] (10,3) circle (2pt) node[below right] {$\Q$};
\filldraw[black] (7,3) circle (2pt) node[left] {$\Q$};
\filldraw[black] (13,3) circle (2pt) node[right] {$\Q$};
\filldraw[black] (10,6) circle (2pt) node[right] {$\Q$};
\filldraw[black] (7,6) circle (2pt) node[left] {$\Q$};
\filldraw[black] (13,6) circle (2pt) node[right] {$\Q$};
\filldraw[black] (10,9) circle (2pt) node[above] {$\Q$};
\draw[ultra thick, <-](10,0.2)--(10,2.8);
\draw[ultra thick, <-](9.8,0.2)--(7.2,2.8);
\draw[ultra thick, <-](10.2,0.2)--(12.8,2.8);
\draw[ultra thick, <-](10,3.2)--(10,5.8);
\draw[ultra thick, <-](7,3.2)--(7,5.8);
\draw[ultra thick, <-](13,3.2)--(13,5.8);
\draw[ultra thick, <-](9.8, 3.2)--(7.2, 5.8);
\draw[ultra thick, <-](10.2,3.2)--(12.8,5.8);
\draw[ultra thick, <-](7.2,6.2)--(9.8,8.8);
\draw[ultra thick, <-](12.8,6.2)--(10.2,8.8);
\draw[ultra thick, <-](10,6.2)--(10,8.8);
\draw [ultra thick, ->,red] (0,0.2) to [out=10,in=170] (10,0.2);
\draw [ultra thick, ->,red] (0,2.8) to [out=-10,in=-170] (10,2.8);
\draw [ultra thick, ->,red] (0,5.8) to [out=-10,in=-170] (10,5.8);
\draw [ultra thick, ->,red] (0,9.2) to [out=10,in=170] (10,9.2);
\draw [ultra thick, ->,red] (-3,3.2) to [out=10,in=170] (7,3.2);
\draw [ultra thick, ->,red] (3,3.2) to [out=10,in=170] (13,3.2);
\draw [ultra thick, ->,red] (-3,6.2) to [out=10,in=170] (7,6.2);
\draw [ultra thick, ->,red] (3,6.2) to [out=10,in=170] (13,6.2);
\node[red] at (5,0.3) {$f_0$};
\node[red] at (2,4.2) {$f_1$};
\node[red] at (5,2.7) {$f_2$};
\node[red] at (8,4.2) {$f_3$};
\node[red] at (1,7.2) {$f_4$};
\node[red] at (5,5.7) {$f_5$};
\node[red] at (9,7.2) {$f_6$};
\node[red] at (5,9.3) {$f_7$};
\end{tikzpicture}
\caption{$\Hom_{\cg^{\Q}}(\underline{H_0}(Z(K;(D^1, S^0)), \underline{\Q})$} 
\label{D_8H_0Q}
\end{figure} 
\begin{enumerate}
\item
Starting from the map $f_0$, it could be any $r\in \Q$;
\item
Then by Lemma \ref{handy}, $f_1$, $f_2$ and $f_3$ are uniquely determined by $f_0$;
\begin{figure}[H]
\centering
\begin{tikzpicture}[scale=1]
\node[] at (0,0) {$\Q$};
\node[] at (0,2) {$\Q$};
\node[] at (2,0) {$\Q$};
\node[] at (2,2) {$\Q$};
\node[above] at (1,0) {$f_0=r$};
\node[above] at (1,2) {$f_1$};
\node[left] at (0,1) {$1$};
\node[right] at (2,1) {$1$};
\draw[ultra thick, ->](0,1.7)--(0,0.3);
\draw[ultra thick, ->](2,1.7)--(2,0.3);
\draw[ultra thick, ->](0.3,0)--(1.7,0);
\draw[ultra thick, ->](0.3,2)--(1.7,2);
\end{tikzpicture}
\end{figure}
\item
The map $f_4$ is uniquely determined by $f_2$;
\begin{figure}[H]
\centering
\begin{tikzpicture}[scale=1]
\node[] at (0,0) {$\Q$};
\node[] at (0,2) {$\Q$};
\node[] at (2,0) {$\Q$};
\node[] at (2,2) {$\Q$};
\node[above] at (1,0) {$f_2$};
\node[above] at (1,2) {$f_4$};
\node[left] at (0,1) {$1$};
\node[right] at (2,1) {$1$};
\draw[ultra thick, ->](0,1.7)--(0,0.3);
\draw[ultra thick, ->](2,1.7)--(2,0.3);
\draw[ultra thick, ->](0.3,0)--(1.7,0);
\draw[ultra thick, ->](0.3,2)--(1.7,2);
\end{tikzpicture}
\end{figure}
\item
The maps $f_5$ and $f_6$ are uniquely determined by $f_2$ if $f_2=f_0=r\neq 0$;
\begin{figure}[H]
\centering
\begin{tikzpicture}[scale=1]
\node[] at (0,0) {$\Q$};
\node[] at (0,2) {$\Q^2$};
\node[] at (2,0) {$\Q$};
\node[] at (2,2) {$\Q$};
\node[above] at (1,0) {$f_2=r$};
\node[above] at (1,2) {$f_5$};
\node[left] at (0,1) {$(1,-1)$};
\node[right] at (2,1) {$1$};
\draw[ultra thick, ->](0,1.7)--(0,0.3);
\draw[ultra thick, ->](2,1.7)--(2,0.3);
\draw[ultra thick, ->](0.3,0)--(1.7,0);
\draw[ultra thick, ->](0.3,2)--(1.7,2);
\end{tikzpicture}
\end{figure}
\item
We have $f_5=(s,s)$ and $f_6=(t,t)$ if $f_2=r=0$ for some $s,t\in \Q$;
\begin{figure}[H]
\centering
\begin{tikzpicture}[scale=1]
\node[] at (0,0) {$\Q$};
\node[] at (0,2) {$\Q^2$};
\node[] at (2,0) {$\Q$};
\node[] at (2,2) {$\Q$};
\node[above] at (1,0) {$f_2=0$};
\node[above] at (1,2) {$f_5=(s,s)$};
\node[left] at (0,1) {$(1,-1)$};
\node[right] at (2,1) {$1$};
\draw[ultra thick, ->](0,1.7)--(0,0.3);
\draw[ultra thick, ->](2,1.7)--(2,0.3);
\draw[ultra thick, ->](0.3,0)--(1.7,0);
\draw[ultra thick, ->](0.3,2)--(1.7,2);
\end{tikzpicture}
\end{figure}
\item
The map $f_7$ is uniquely determined by $f_5$. similarly, it is also uniquely determined by $f_6$. Hence we have $s=t$.
\begin{figure}[H]
\centering
\begin{tikzpicture}[scale=1]
\node[] at (0,0) {$\Q^2$};
\node[] at (0,2) {$\Q^2$};
\node[] at (2,0) {$\Q$};
\node[] at (2,2) {$\Q^2$};
\node[above] at (1,0) {$f_5,f_6$};
\node[above] at (1,2) {$f_7$};
\node[left] at (0,1) {Id};
\node[right] at (2,1) {$1$};
\draw[ultra thick, ->](0,1.7)--(0,0.3);
\draw[ultra thick, ->](2,1.7)--(2,0.3);
\draw[ultra thick, ->](0.3,0)--(1.7,0);
\draw[ultra thick, ->](0.3,2)--(1.7,2);
\end{tikzpicture}
\end{figure}
\item
In summary, there are two free variables. $$\Hom_{\cg^{\Q}}(\underline{H_0}(Z(K;(D^1, S^0)), \underline{\Q})=\Q^2.$$
\end{enumerate}

Secondly, we move on to compute $$H^1_{D_8}(Z(K;(D^1, S^0)),\underline{\Q})= \Hom_{\cg^{\Q}}(\underline{H_1}(Z(K;(D^1, S^0)), \underline{\Q}).$$ The commutative diagram is shown in Figure \ref{D_8H_1Q}.
\begin{figure}[h]
\centering
\begin{tikzpicture}[scale=0.7]
\filldraw[black] (0,0) circle (2pt) node[below] {$0$};
\filldraw[black] (0,3) circle (2pt) node[below right] {$\Q$};
\filldraw[black] (-3,3) circle (2pt) node[left] {$0$};
\filldraw[black] (3,3) circle (2pt) node[right] {$\Q$};
\filldraw[black] (0,6) circle (2pt) node[right] {$0$};
\filldraw[black] (-3,6) circle (2pt) node[left] {$\Q$};
\filldraw[black] (3,6) circle (2pt) node[right] {$0$};
\filldraw[black] (0,9) circle (2pt) node[above] {$0$};
\draw[ultra thick, <-](0,0.2)--(0,2.8);
\draw[ultra thick, <-](-0.2,0.2)--(-2.8,2.8);
\draw[ultra thick, <-](0.2,0.2)--(2.8,2.8);
\draw[ultra thick, <-](0,3.2)--(0,5.8);
\draw[ultra thick, <-](-3,3.2)--(-3,5.8);
\draw[ultra thick, <-](3,3.2)--(3,5.8);
\draw[ultra thick, <-](-0.2, 3.2)--(-2.8, 5.8);
\draw[ultra thick, <-](0.2,3.2)--(2.8,5.8);
\draw[ultra thick, <-](-2.8,6.2)--(-0.2,8.8);
\draw[ultra thick, <-](2.8,6.2)--(0.2,8.8);
\draw[ultra thick, <-](0,6.2)--(0,8.8);
\filldraw[black] (10,0) circle (2pt) node[below] {$\Q$};
\filldraw[black] (10,3) circle (2pt) node[below right] {$\Q$};
\filldraw[black] (7,3) circle (2pt) node[left] {$\Q$};
\filldraw[black] (13,3) circle (2pt) node[right] {$\Q$};
\filldraw[black] (10,6) circle (2pt) node[right] {$\Q$};
\filldraw[black] (7,6) circle (2pt) node[left] {$\Q$};
\filldraw[black] (13,6) circle (2pt) node[right] {$\Q$};
\filldraw[black] (10,9) circle (2pt) node[above] {$\Q$};
\draw[ultra thick, <-](10,0.2)--(10,2.8);
\draw[ultra thick, <-](9.8,0.2)--(7.2,2.8);
\draw[ultra thick, <-](10.2,0.2)--(12.8,2.8);
\draw[ultra thick, <-](10,3.2)--(10,5.8);
\draw[ultra thick, <-](7,3.2)--(7,5.8);
\draw[ultra thick, <-](13,3.2)--(13,5.8);
\draw[ultra thick, <-](9.8, 3.2)--(7.2, 5.8);
\draw[ultra thick, <-](10.2,3.2)--(12.8,5.8);
\draw[ultra thick, <-](7.2,6.2)--(9.8,8.8);
\draw[ultra thick, <-](12.8,6.2)--(10.2,8.8);
\draw[ultra thick, <-](10,6.2)--(10,8.8);
\draw [ultra thick, ->,red] (0,2.8) to [out=-10,in=-170] (10,2.8);
\draw [ultra thick, ->,red] (3,3.2) to [out=10,in=170] (13,3.2);
\draw [ultra thick, ->,red] (-3,6.2) to [out=10,in=170] (7,6.2);
\node[red] at (5,2.7) {$f_2$};
\node[red] at (8,4.2) {$f_3$};
\node[red] at (1,7.2) {$f_4$};
\end{tikzpicture}
\caption{$\Hom_{\cg^{\Q}}(\underline{H_1}(Z(K;(D^1, S^0)), \underline{\Q})$} 
\label{D_8H_1Q}
\end{figure} 
The possibly nonzero maps are $f_2$, $f_3$ and $f_4$. However, by Lemma \ref{handy}, $f_2=f_3=f_4=0$. 
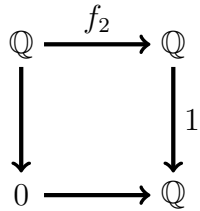
\begin{figure}[H]
\centering
\begin{tikzpicture}[scale=1]
\node[] at (0,0) {$0$};
\node[] at (0,2) {$\Q$};
\node[] at (2,0) {$\Q$};
\node[] at (2,2) {$\Q$};
\node[above] at (1,2) {$f_2$};
\node[right] at (2,1) {$1$};
\draw[ultra thick, ->](0,1.7)--(0,0.3);
\draw[ultra thick, ->](2,1.7)--(2,0.3);
\draw[ultra thick, ->](0.3,0)--(1.7,0);
\draw[ultra thick, ->](0.3,2)--(1.7,2);
\end{tikzpicture}
\end{figure}
We have 
$$\Hom_{\cg^{\Q}}(\underline{H_1}(Z(K;(D^1, S^0)), \underline{\Q})=0.$$

Thirdly, let us compute $$H^2_{D_8}(Z(K;(D^1, S^0)),\underline{\Q})= \Hom_{\cg^{\Q}}(\underline{H_2}(Z(K;(D^1, S^0)), \underline{\Q}).$$ The commutative diagram is shown in Figure \ref{D_8H_2Q}.
\begin{figure}[h]
\centering
\begin{tikzpicture}[scale=0.7]
\filldraw[black] (0,0) circle (2pt) node[below] {$0$};
\filldraw[black] (0,3) circle (2pt) node[below right] {$0$};
\filldraw[black] (-3,3) circle (2pt) node[left] {$\Q$};
\filldraw[black] (3,3) circle (2pt) node[right] {$0$};
\filldraw[black] (0,6) circle (2pt) node[right] {$0$};
\filldraw[black] (-3,6) circle (2pt) node[left] {$0$};
\filldraw[black] (3,6) circle (2pt) node[right] {$0$};
\filldraw[black] (0,9) circle (2pt) node[above] {$0$};
\draw[ultra thick, <-](0,0.2)--(0,2.8);
\draw[ultra thick, <-](-0.2,0.2)--(-2.8,2.8);
\draw[ultra thick, <-](0.2,0.2)--(2.8,2.8);
\draw[ultra thick, <-](0,3.2)--(0,5.8);
\draw[ultra thick, <-](-3,3.2)--(-3,5.8);
\draw[ultra thick, <-](3,3.2)--(3,5.8);
\draw[ultra thick, <-](-0.2, 3.2)--(-2.8, 5.8);
\draw[ultra thick, <-](0.2,3.2)--(2.8,5.8);
\draw[ultra thick, <-](-2.8,6.2)--(-0.2,8.8);
\draw[ultra thick, <-](2.8,6.2)--(0.2,8.8);
\draw[ultra thick, <-](0,6.2)--(0,8.8);
\filldraw[black] (10,0) circle (2pt) node[below] {$\Q$};
\filldraw[black] (10,3) circle (2pt) node[below right] {$\Q$};
\filldraw[black] (7,3) circle (2pt) node[left] {$\Q$};
\filldraw[black] (13,3) circle (2pt) node[right] {$\Q$};
\filldraw[black] (10,6) circle (2pt) node[right] {$\Q$};
\filldraw[black] (7,6) circle (2pt) node[left] {$\Q$};
\filldraw[black] (13,6) circle (2pt) node[right] {$\Q$};
\filldraw[black] (10,9) circle (2pt) node[above] {$\Q$};
\draw[ultra thick, <-](10,0.2)--(10,2.8);
\draw[ultra thick, <-](9.8,0.2)--(7.2,2.8);
\draw[ultra thick, <-](10.2,0.2)--(12.8,2.8);
\draw[ultra thick, <-](10,3.2)--(10,5.8);
\draw[ultra thick, <-](7,3.2)--(7,5.8);
\draw[ultra thick, <-](13,3.2)--(13,5.8);
\draw[ultra thick, <-](9.8, 3.2)--(7.2, 5.8);
\draw[ultra thick, <-](10.2,3.2)--(12.8,5.8);
\draw[ultra thick, <-](7.2,6.2)--(9.8,8.8);
\draw[ultra thick, <-](12.8,6.2)--(10.2,8.8);
\draw[ultra thick, <-](10,6.2)--(10,8.8);
\draw [ultra thick, ->,red] (-3,3.2) to [out=10,in=170] (7,3.2);
\node[red] at (2,4.2) {$f_1$};

\end{tikzpicture}
\caption{$\Hom_{\cg^{\Q}}(\underline{H_2}(Z(K;(D^1, S^0)), \underline{\Q})$} 
\label{D_8H_2Q}
\end{figure}  

Again, by Lemma \ref{handy}, $f_1=0$.

Finally, for $$H^3_{D_8}(Z(K;(D^1, S^0)),\underline{\Q})= \Hom_{\cg^{\Q}}(\underline{H_3}(Z(K;(D^1, S^0)), \underline{\Q})=\Q.$$ Since the map $f_0:\Q\rightarrow \Q$ could be any rational number. The commutative diagram is shown in Figure \ref{D_8H_3Q}.
\begin{figure}[h]
\centering
\begin{tikzpicture}[scale=0.7]
\filldraw[black] (0,0) circle (2pt) node[below] {$\Q$};
\filldraw[black] (0,3) circle (2pt) node[below right] {$0$};
\filldraw[black] (-3,3) circle (2pt) node[left] {$0$};
\filldraw[black] (3,3) circle (2pt) node[right] {$0$};
\filldraw[black] (0,6) circle (2pt) node[right] {$0$};
\filldraw[black] (-3,6) circle (2pt) node[left] {$0$};
\filldraw[black] (3,6) circle (2pt) node[right] {$0$};
\filldraw[black] (0,9) circle (2pt) node[above] {$0$};
\draw[ultra thick, <-](0,0.2)--(0,2.8);
\draw[ultra thick, <-](-0.2,0.2)--(-2.8,2.8);
\draw[ultra thick, <-](0.2,0.2)--(2.8,2.8);
\draw[ultra thick, <-](0,3.2)--(0,5.8);
\draw[ultra thick, <-](-3,3.2)--(-3,5.8);
\draw[ultra thick, <-](3,3.2)--(3,5.8);
\draw[ultra thick, <-](-0.2, 3.2)--(-2.8, 5.8);
\draw[ultra thick, <-](0.2,3.2)--(2.8,5.8);
\draw[ultra thick, <-](-2.8,6.2)--(-0.2,8.8);
\draw[ultra thick, <-](2.8,6.2)--(0.2,8.8);
\draw[ultra thick, <-](0,6.2)--(0,8.8);
\filldraw[black] (10,0) circle (2pt) node[below] {$\Q$};
\filldraw[black] (10,3) circle (2pt) node[below right] {$\Q$};
\filldraw[black] (7,3) circle (2pt) node[left] {$\Q$};
\filldraw[black] (13,3) circle (2pt) node[right] {$\Q$};
\filldraw[black] (10,6) circle (2pt) node[right] {$\Q$};
\filldraw[black] (7,6) circle (2pt) node[left] {$\Q$};
\filldraw[black] (13,6) circle (2pt) node[right] {$\Q$};
\filldraw[black] (10,9) circle (2pt) node[above] {$\Q$};
\draw[ultra thick, <-](10,0.2)--(10,2.8);
\draw[ultra thick, <-](9.8,0.2)--(7.2,2.8);
\draw[ultra thick, <-](10.2,0.2)--(12.8,2.8);
\draw[ultra thick, <-](10,3.2)--(10,5.8);
\draw[ultra thick, <-](7,3.2)--(7,5.8);
\draw[ultra thick, <-](13,3.2)--(13,5.8);
\draw[ultra thick, <-](9.8, 3.2)--(7.2, 5.8);
\draw[ultra thick, <-](10.2,3.2)--(12.8,5.8);
\draw[ultra thick, <-](7.2,6.2)--(9.8,8.8);
\draw[ultra thick, <-](12.8,6.2)--(10.2,8.8);
\draw[ultra thick, <-](10,6.2)--(10,8.8);
\draw [ultra thick, ->,red] (0,0.2) to [out=10,in=170] (10,0.2);
\node[red] at (5,0.3) {$f_0$};
\end{tikzpicture}
\caption{$\Hom_{\cg^{\Q}}(\underline{H_3}(Z(K;(D^1, S^0)), \underline{\Q})$} 
\label{D_8H_3Q}
\end{figure} 

\subsection{Computation for $G=\Sigma_4$}

Now let us consider the full automorphism group $\Sigma_4=\Aut(\partial \Delta^3)$. It is the symmetric group of degree 4. The computation for Bredon cohomology is almost the same but with one more layer of subgroups. The canonical orbit category is shown in the left diagram in Figure \ref{fig:S4}. $W0=G$, $W\langle \sigma^2\rangle=\Z_2\times Z_2$, $WG=WD_8=W\sigma_3=0$ and all other subgroups has Weyl group $\Z_2$. Again, for the convenience of future reference, we name the objects in $\og$ as in the right diagram of Figure \ref{fig:S4}.

\begin{figure}[h]
\centering
\begin{tikzpicture}[scale=0.5]
\filldraw[black] (0,0) circle (2pt) node[below] {$G/\{e\}$};
\filldraw[black] (-2,3) circle (2pt) node[right] {$G/\langle\sigma^2\rangle$};
\filldraw[black] (-5,3) circle (2pt) node[left] {$G/\langle\tau\sigma\rangle$};
\filldraw[black] (2,3) circle (2pt) node[right] {$G/\langle\tau\rangle$};
\filldraw[black] (-2,6) circle (2pt) node[right] {$G/\langle\sigma\rangle$};
\filldraw[black] (-5,6) circle (2pt) node[left] {$G/\langle\tau\sigma,\sigma^2\rangle$};
\filldraw[black] (2,6) circle (2pt) node[right] {$G/\langle\tau, \sigma^2\rangle$};
\filldraw[black] (-2,9) circle (2pt) node[above left] {$G/D_8$};
\filldraw[black] (1,10) circle (2pt) node[right] {$G/A_4$};
\filldraw[black] (0,12) circle (2pt) node[above] {$G/G$};
\filldraw[black] (5,4) circle (2pt) node[right] {$G/\langle(123)\rangle$};
\filldraw[black] (5,8) circle (2pt) node[right] {$G/\Sigma_3$};
\draw[ultra thick, ->](-0.1,0.2)--(-1.9,2.8);
\draw[ultra thick, ->](-0.3,0.2)--(-4.8,2.8);
\draw[ultra thick, ->](0.1,0.2)--(1.9,2.8);
\draw[ultra thick, ->](0.3,0.2)--(4.8,3.8);
\draw[ultra thick, ->](-5,3.2)--(-5,5.8);
\draw[ultra thick, ->](-2,3.2)--(-2,5.8);
\draw[ultra thick, ->](-1.8,3.2)--(1.8,5.8);
\draw[ultra thick, ->](-2.2,3.2)--(-4.8,5.8);
\draw[ultra thick, ->](2,3.2)--(2,5.8);
\draw[ultra thick, ->](5,4.2)--(5, 7.8);
\draw[ultra thick, ->](-4.8,6.2)--(-2.2,8.8);
\draw[ultra thick, ->](-2,6.2)--(-2,8.8);
\draw[ultra thick, ->](1.8,6.2)--(-1.8,8.8);
\draw[ultra thick, ->](2,6.2)--(1,9.8);
\draw[ultra thick, ->](-2,9.2)--(-0.2,11.8);
\draw[ultra thick, ->](0.9,10.1)--(0.1,11.8);
\draw[ultra thick, ->](4.8,8.2)--(0.2,12);
\filldraw[black] (14,0) circle (2pt) node[below] {0};
\filldraw[black] (12,3) circle (2pt) node[below left] {2};
\filldraw[black] (9,3) circle (2pt) node[left] {1};
\filldraw[black] (16,3) circle (2pt) node[right] {3};
\filldraw[black] (12,6) circle (2pt) node[right] {5};
\filldraw[black] (9,6) circle (2pt) node[left] {4};
\filldraw[black] (16,6) circle (2pt) node[right] {6};
\filldraw[black] (12,9) circle (2pt) node[above left] {7};
\filldraw[black] (15,10) circle (2pt) node[right] {10};
\filldraw[black] (14,12) circle (2pt) node[above] {11};
\filldraw[black] (19,4) circle (2pt) node[right] {8};
\filldraw[black] (19,8) circle (2pt) node[right] {9};
\draw[ultra thick, ->](13.9,0.2)--(12.1,2.8);
\draw[ultra thick, ->](13.7,0.2)--(9.2,2.8);
\draw[ultra thick, ->](14.1,0.2)--(15.9,2.8);
\draw[ultra thick, ->](14.3,0.2)--(18.8,3.8);
\draw[ultra thick, ->](9,3.2)--(9,5.8);
\draw[ultra thick, ->](12,3.2)--(12,5.8);
\draw[ultra thick, ->](12.2,3.2)--(15.8,5.8);
\draw[ultra thick, ->](11.8,3.2)--(9.2,5.8);
\draw[ultra thick, ->](16,3.2)--(16,5.8);
\draw[ultra thick, ->](19,4.2)--(19, 7.8);
\draw[ultra thick, ->](9.2,6.2)--(11.8,8.8);
\draw[ultra thick, ->](12,6.2)--(12,8.8);
\draw[ultra thick, ->](15.8,6.2)--(12.2,8.8);
\draw[ultra thick, ->](16,6.2)--(15,9.8);
\draw[ultra thick, ->](12,9.2)--(13.8,11.8);
\draw[ultra thick, ->](14.9,10.1)--(14.1,11.8);
\draw[ultra thick, ->](18.8,8.2)--(14.2,12);
\end{tikzpicture}
\caption{$\mathscr{C}_{\Sigma_4}$} 
\label{fig:S4}
\end{figure}

Similar to $D_8$ case, we have the fixed point data for any subgroup of $G$, shown in Figure \ref{fig:S3S4}.

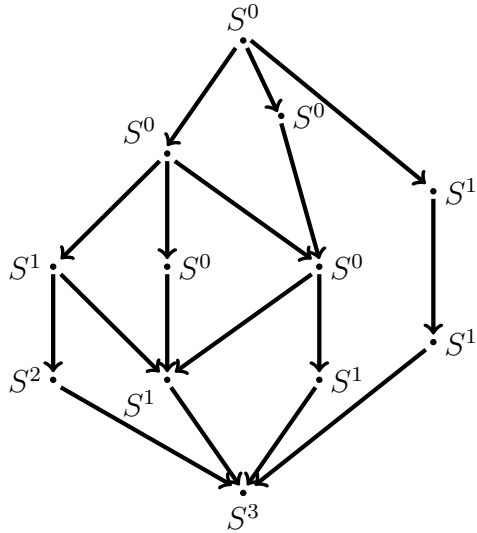
\begin{figure}[h]
\centering
\begin{tikzpicture}[scale=0.5]
\filldraw[black] (0,0) circle (2pt) node[below] {$S^3$};
\filldraw[black] (-2,3) circle (2pt) node[below left] {$S^1$};
\filldraw[black] (-5,3) circle (2pt) node[left] {$S^2$};
\filldraw[black] (2,3) circle (2pt) node[right] {$S^1$};
\filldraw[black] (-2,6) circle (2pt) node[right] {$S^0$};
\filldraw[black] (-5,6) circle (2pt) node[left] {$S^1$};
\filldraw[black] (2,6) circle (2pt) node[right] {$S^0$};
\filldraw[black] (-2,9) circle (2pt) node[above left] {$S^0$};
\filldraw[black] (1,10) circle (2pt) node[right] {$S^0$};
\filldraw[black] (0,12) circle (2pt) node[above] {$S^0$};
\filldraw[black] (5,4) circle (2pt) node[right] {$S^1$};
\filldraw[black] (5,8) circle (2pt) node[right] {$S^1$};
\draw[ultra thick, <-](-0.1,0.2)--(-1.9,2.8);
\draw[ultra thick, <-](-0.3,0.2)--(-4.8,2.8);
\draw[ultra thick, <-](0.1,0.2)--(1.9,2.8);
\draw[ultra thick, <-](0.3,0.2)--(4.8,3.8);
\draw[ultra thick, <-](-5,3.2)--(-5,5.8);
\draw[ultra thick, <-](-2,3.2)--(-2,5.8);
\draw[ultra thick, <-](-1.8,3.2)--(1.8,5.8);
\draw[ultra thick, <-](-2.2,3.2)--(-4.8,5.8);
\draw[ultra thick, <-](2,3.2)--(2,5.8);
\draw[ultra thick, <-](5,4.2)--(5, 7.8);
\draw[ultra thick, <-](-4.8,6.2)--(-2.2,8.8);
\draw[ultra thick, <-](-2,6.2)--(-2,8.8);
\draw[ultra thick, <-](1.8,6.2)--(-1.8,8.8);
\draw[ultra thick, <-](2,6.2)--(1,9.8);
\draw[ultra thick, <-](-2,9.2)--(-0.2,11.8);
\draw[ultra thick, <-](0.9,10.1)--(0.1,11.8);
\draw[ultra thick, <-](4.8,8.2)--(0.2,12);
\end{tikzpicture}
\caption{Fixed point set of $S^3$ with $\Sigma_4$-action} 
\label{fig:S3S4}
\end{figure}

Firstly, for $H^0_{\Sigma_4}(Z(K;(D^1, S^0)), \underline{\Q})=  \Hom_{\cg^{\Q}}(\underline{H_0}(Z(K;(D^1, S^0)), \underline{\Q})$. The commutative diagram is shown in Figure \ref{S_4H_0Q}.
\begin{figure}[h]
\centering
\begin{tikzpicture}[scale=0.6]
\filldraw[black] (0,0) node[below] {$\Q$};
\filldraw[black] (-2,3) circle (2pt) node[below left] {$\Q$};
\filldraw[black] (-5,3) circle (2pt) node[left] {$\Q$};
\filldraw[black] (2,3) circle (2pt) node[right] {$\Q$};
\filldraw[black] (-2,6) circle (2pt) node[right] {$\Q^2$};
\filldraw[black] (-5,6) circle (2pt) node[left] {$\Q$};
\filldraw[black] (2,6) circle (2pt) node[right] {$\Q^2$};
\filldraw[black] (-2,9) circle (2pt) node[above left] {$\Q^2$};
\filldraw[black] (1,10) circle (2pt) node[right] {$\Q^2$};
\filldraw[black] (0,12) circle (2pt) node[above] {$\Q^2$};
\filldraw[black] (5,4) circle (2pt) node[right] {$\Q$};
\filldraw[black] (5,8) circle (2pt) node[right] {$\Q$};
\draw[ultra thick, <-](-0.1,0.2)--(-1.9,2.8);
\draw[ultra thick, <-](-0.3,0.2)--(-4.8,2.8);
\draw[ultra thick, <-](0.1,0.2)--(1.9,2.8);
\draw[ultra thick, <-](0.3,0.2)--(4.8,3.8);
\draw[ultra thick, <-](-5,3.2)--(-5,5.8);
\draw[ultra thick, <-](-2,3.2)--(-2,5.8);
\draw[ultra thick, <-](-1.8,3.2)--(1.8,5.8);
\draw[ultra thick, <-](-2.2,3.2)--(-4.8,5.8);
\draw[ultra thick, <-](2,3.2)--(2,5.8);
\draw[ultra thick, <-](5,4.2)--(5, 7.8);
\draw[ultra thick, <-](-4.8,6.2)--(-2.2,8.8);
\draw[ultra thick, <-](-2,6.2)--(-2,8.8);
\draw[ultra thick, <-](1.8,6.2)--(-1.8,8.8);
\draw[ultra thick, <-](2,6.2)--(1,9.8);
\draw[ultra thick, <-](-2,9.2)--(-0.2,11.8);
\draw[ultra thick, <-](0.9,10.1)--(0.1,11.8);
\draw[ultra thick, <-](4.8,8.2)--(0.2,12);
\filldraw[black] (14,0) circle (2pt) node[below] {$\Q$};
\filldraw[black] (12,3) circle (2pt) node[below left] {$\Q$};
\filldraw[black] (9,3) circle (2pt) node[left] {$\Q$};
\filldraw[black] (16,3) circle (2pt) node[right] {$\Q$};
\filldraw[black] (12,6) circle (2pt) node[right] {$\Q$};
\filldraw[black] (9,6) circle (2pt) node[left] {$\Q$};
\filldraw[black] (16,6) circle (2pt) node[right] {$\Q$};
\filldraw[black] (12,9) circle (2pt) node[above left] {$\Q$};
\filldraw[black] (15,10) circle (2pt) node[right] {$\Q$};
\filldraw[black] (14,12) circle (2pt) node[above] {$\Q$};
\filldraw[black] (19,4) circle (2pt) node[right] {$\Q$};
\filldraw[black] (19,8) circle (2pt) node[right] {$\Q$};
\draw[ultra thick, <-](13.9,0.2)--(12.1,2.8);
\draw[ultra thick, <-](13.7,0.2)--(9.2,2.8);
\draw[ultra thick, <-](14.1,0.2)--(15.9,2.8);
\draw[ultra thick, <-](14.3,0.2)--(18.8,3.8);
\draw[ultra thick, <-](9,3.2)--(9,5.8);
\draw[ultra thick, <-](12,3.2)--(12,5.8);
\draw[ultra thick, <-](12.2,3.2)--(15.8,5.8);
\draw[ultra thick, <-](11.8,3.2)--(9.2,5.8);
\draw[ultra thick, <-](16,3.2)--(16,5.8);
\draw[ultra thick, <-](19,4.2)--(19, 7.8);
\draw[ultra thick, <-](9.2,6.2)--(11.8,8.8);
\draw[ultra thick, <-](12,6.2)--(12,8.8);
\draw[ultra thick, <-](15.8,6.2)--(12.2,8.8);
\draw[ultra thick, <-](16,6.2)--(15,9.8);
\draw[ultra thick, <-](12,9.2)--(13.8,11.8);
\draw[ultra thick, <-](14.9,10.1)--(14.1,11.8);
\draw[ultra thick, <-](18.8,8.2)--(14.2,12);
\draw [ultra thick, ->,red] (0,0.2) to [out=10,in=170] (14,0.2);
\draw [ultra thick, ->,red] (-2,3.2) to [out=5,in=175] (12,3.2);
\draw [ultra thick, ->,red] (-5,2.8) to [out=-10,in=-170] (9,2.8);
\draw [ultra thick, ->,red] (2,2.8) to [out=-10,in=-170] (16,2.8);
\draw [ultra thick, ->,red] (-2,6.2) to [out=10,in=170] (12,6.2);
\draw [ultra thick, ->,red] (-5,5.8) to [out=-10,in=-170] (9,5.8);
\draw [ultra thick, ->,red] (2,5.8) to [out=-10,in=-170] (16,5.8);
\draw [ultra thick, ->,red] (-2,9.2) to [out=10,in=170] (12,9.2);
\draw [ultra thick, ->,red] (1,10.2) to [out=10,in=170] (15,10.2);
\draw [ultra thick, ->,red] (0,12.2) to [out=10,in=170] (14,12.2);
\draw [ultra thick, ->,red] (5.2,4) to (18.8,4);
\draw [ultra thick, ->,red] (5.2,8) to (18.8,8);
\end{tikzpicture}
\caption{$\Hom_{\cg^{\Q}}(\underline{H_0}(Z(K;(D^1, S^0)), \underline{\Q})$} 
\label{S_4H_0Q}
\end{figure}
\begin{enumerate}
\item
The map $f_0$ could be any $r\in \Q$.
\item
The maps $f_1$, $f_2$, $f_3$ and $f_8$ are uniquely determined by $f_0$.
\item
The maps $f_4$ is uniquely determined by $f_2$.
\item
The maps $f_9$ is uniquely determined by $f_8$.
\item
The maps $f_5$, $f_6$ are uniquely determined by $f_2$ if $r\neq 0$.
\item
If $r=0$, $f_5=f_6=f_7=(s,s)$ for some $s\in \Q$.
\item
The map $f_{10}$ is uniquely determined by $f_6$.
\item
The map $f_{11}$ is uniquely determined by $f_7$ or $f_10$. And since $f_7=f_{10}=(s,s)$, hence there is no further restriction on $s$.
\item
In summary, 
$$\Hom_{\cg^{\Q}}(\underline{H_0}(Z(K;(D^1, S^0)), \underline{\Q})=\Q^2.$$
\end{enumerate}

Secondly, for $$H^1_{\Sigma_4}(Z(K;(D^1, S^0)),\underline{\Q})= \Hom_{\cg^{\Q}}(\underline{H_1}(Z(K;(D^1, S^0)), \underline{\Q}),$$ the commutative diagram is shown in Figure \ref{S_4H_1Q}.

\begin{figure}[h]
\centering
\begin{tikzpicture}[scale=0.6]
\filldraw[black] (0,0) node[below] {$0$};
\filldraw[black] (-2,3) circle (2pt) node[below left] {$\Q$};
\filldraw[black] (-5,3) circle (2pt) node[left] {$0$};
\filldraw[black] (2,3) circle (2pt) node[right] {$\Q$};
\filldraw[black] (-2,6) circle (2pt) node[right] {$0$};
\filldraw[black] (-5,6) circle (2pt) node[left] {$\Q$};
\filldraw[black] (2,6) circle (2pt) node[right] {$0$};
\filldraw[black] (-2,9) circle (2pt) node[above left] {$0$};
\filldraw[black] (1,10) circle (2pt) node[right] {$0$};
\filldraw[black] (0,12) circle (2pt) node[above] {$0$};
\filldraw[black] (5,4) circle (2pt) node[right] {$\Q$};
\filldraw[black] (5,8) circle (2pt) node[right] {$\Q$};
\draw[ultra thick, <-](-0.1,0.2)--(-1.9,2.8);
\draw[ultra thick, <-](-0.3,0.2)--(-4.8,2.8);
\draw[ultra thick, <-](0.1,0.2)--(1.9,2.8);
\draw[ultra thick, <-](0.3,0.2)--(4.8,3.8);
\draw[ultra thick, <-](-5,3.2)--(-5,5.8);
\draw[ultra thick, <-](-2,3.2)--(-2,5.8);
\draw[ultra thick, <-](-1.8,3.2)--(1.8,5.8);
\draw[ultra thick, <-](-2.2,3.2)--(-4.8,5.8);
\draw[ultra thick, <-](2,3.2)--(2,5.8);
\draw[ultra thick, <-](5,4.2)--(5, 7.8);
\draw[ultra thick, <-](-4.8,6.2)--(-2.2,8.8);
\draw[ultra thick, <-](-2,6.2)--(-2,8.8);
\draw[ultra thick, <-](1.8,6.2)--(-1.8,8.8);
\draw[ultra thick, <-](2,6.2)--(1,9.8);
\draw[ultra thick, <-](-2,9.2)--(-0.2,11.8);
\draw[ultra thick, <-](0.9,10.1)--(0.1,11.8);
\draw[ultra thick, <-](4.8,8.2)--(0.2,12);
\filldraw[black] (14,0) circle (2pt) node[below] {$\Q$};
\filldraw[black] (12,3) circle (2pt) node[below left] {$\Q$};
\filldraw[black] (9,3) circle (2pt) node[left] {$\Q$};
\filldraw[black] (16,3) circle (2pt) node[right] {$\Q$};
\filldraw[black] (12,6) circle (2pt) node[right] {$\Q$};
\filldraw[black] (9,6) circle (2pt) node[left] {$\Q$};
\filldraw[black] (16,6) circle (2pt) node[right] {$\Q$};
\filldraw[black] (12,9) circle (2pt) node[above left] {$\Q$};
\filldraw[black] (15,10) circle (2pt) node[right] {$\Q$};
\filldraw[black] (14,12) circle (2pt) node[above] {$\Q$};
\filldraw[black] (19,4) circle (2pt) node[right] {$\Q$};
\filldraw[black] (19,8) circle (2pt) node[right] {$\Q$};
\draw[ultra thick, <-](13.9,0.2)--(12.1,2.8);
\draw[ultra thick, <-](13.7,0.2)--(9.2,2.8);
\draw[ultra thick, <-](14.1,0.2)--(15.9,2.8);
\draw[ultra thick, <-](14.3,0.2)--(18.8,3.8);
\draw[ultra thick, <-](9,3.2)--(9,5.8);
\draw[ultra thick, <-](12,3.2)--(12,5.8);
\draw[ultra thick, <-](12.2,3.2)--(15.8,5.8);
\draw[ultra thick, <-](11.8,3.2)--(9.2,5.8);
\draw[ultra thick, <-](16,3.2)--(16,5.8);
\draw[ultra thick, <-](19,4.2)--(19, 7.8);
\draw[ultra thick, <-](9.2,6.2)--(11.8,8.8);
\draw[ultra thick, <-](12,6.2)--(12,8.8);
\draw[ultra thick, <-](15.8,6.2)--(12.2,8.8);
\draw[ultra thick, <-](16,6.2)--(15,9.8);
\draw[ultra thick, <-](12,9.2)--(13.8,11.8);
\draw[ultra thick, <-](14.9,10.1)--(14.1,11.8);
\draw[ultra thick, <-](18.8,8.2)--(14.2,12);
\draw [ultra thick, ->,red] (-5,5.8) to [out=-10,in=-170] (9,5.8);
\draw [ultra thick, ->,red] (-2,3.2) to [out=5,in=175] (12,3.2);
\draw [ultra thick, ->,red] (2,2.8) to [out=-10,in=-170] (16,2.8);
\draw [ultra thick, ->,red] (5.2,4) to (18.8,4);
\draw [ultra thick, ->,red] (5.2,8) to (18.8,8);
\end{tikzpicture}
\caption{$\Hom_{\cg^{\Q}}(\underline{H_0}(Z(K;(D^1, S^0)), \underline{\Q})$} 
\label{S_4H_1Q}
\end{figure}
\begin{enumerate}
\item
By Lemma \ref{handy}, we have $f_2=f_4=f_8=0$.
\begin{figure}[H]
\centering
\begin{tikzpicture}[scale=1]
\node[] at (0,0) {$0$};
\node[] at (0,2) {$\Q$};
\node[] at (2,0) {$\Q$};
\node[] at (2,2) {$\Q$};
\node[above] at (1,2) {$f_2$};
\node[right] at (2,1) {$1$};
\draw[ultra thick, ->](0,1.7)--(0,0.3);
\draw[ultra thick, ->](2,1.7)--(2,0.3);
\draw[ultra thick, ->](0.3,0)--(1.7,0);
\draw[ultra thick, ->](0.3,2)--(1.7,2);
\end{tikzpicture}
\end{figure}
\item
Then $f_9$ has to be $0$ as well.
\begin{figure}[H]
\centering
\begin{tikzpicture}[scale=1]
\node[] at (0,0) {$\Q$};
\node[] at (0,2) {$\Q$};
\node[] at (2,0) {$\Q$};
\node[] at (2,2) {$\Q$};
\node[above] at (1,2) {$f_9$};
\node[right] at (2,1) {$1$};
\node[above] at (1,0) {$f_8=0$};
\draw[ultra thick, ->](0,1.7)--(0,0.3);
\draw[ultra thick, ->](2,1.7)--(2,0.3);
\draw[ultra thick, ->](0.3,0)--(1.7,0);
\draw[ultra thick, ->](0.3,2)--(1.7,2);
\end{tikzpicture}
\end{figure}
\item
Hence
$$\Hom_{\cg^{\Q}}(\underline{H_1}(Z(K;(D^1, S^0)), \underline{\Q})=0.$$
\end{enumerate}
For $n=2,3$, the computation is exactly the same as in $D_8$ case. Hence we have the same result as well.

Summarize all the computation above, we have the following theorem.

\begin{theorem}
\label{BCforconstantQ}
Let $K$ be the boundary of 3-simplex, the 3-sphere $S^3=Z(K;(D^1, S^0))$ defined by the polyhedral product admits a $\Sigma_4$-action. $$H=\langle (1234),(12)(34)\rangle$$ is one of the Sylow 2-subgroups of $\Sigma_4$. It is also isomorphic to $D_8$, the dehidral group of order 8. For $G=\Sigma_4$ or $D_8$, the Bredon cohomology with coefficient $\underline{\Q}$ is given by
\[
H^n_{G}(Z(K;(D^1, S^0)),\underline{\Q})=\left\{
\begin{array}{ll}
\Q^2 & \text{for } n=0;\\
0 & \text{for } n=1,2;\\
\Q & \text{for } n=3.
\end{array}\right.
\]
\hfill $\square$
\end{theorem}
\subsection{Computation towards general examples}
Finally let us try to compute an example for a coefficient system $\underline{M}$ where $\underline{M}$ is not injective. Again let $K=\partial\Delta^3$ and $Z(K;(D^1,s^0))$ is the $3$-sphere that admits a $\Sigma_4$-action. We are going to restrict ourselves to $G=D_8$ case for this example. Consider the 1-dimension coefficient system $\underline{M}$ given by Figure \ref{fig:BCM}.

\begin{figure}[h]
\centering
\begin{tikzpicture}[scale=0.7]
\filldraw[black] (0,0) circle (2pt) node[below] {$\Q$};
\filldraw[black] (0,3) circle (2pt) node[below right] {$0$};
\filldraw[black] (-3,3) circle (2pt) node[left] {$0$};
\filldraw[black] (3,3) circle (2pt) node[right] {$0$};
\filldraw[black] (0,6) circle (2pt) node[right] {$0$};
\filldraw[black] (-3,6) circle (2pt) node[left] {$0$};
\filldraw[black] (3,6) circle (2pt) node[right] {$0$};
\filldraw[black] (0,9) circle (2pt) node[above] {$0$};
\draw[ultra thick, <-](0,0.2)--(0,2.8);
\draw[ultra thick, <-](-0.2,0.2)--(-2.8,2.8);
\draw[ultra thick, <-](0.2,0.2)--(2.8,2.8);
\draw[ultra thick, <-](0,3.2)--(0,5.8);
\draw[ultra thick, <-](-3,3.2)--(-3,5.8);
\draw[ultra thick, <-](3,3.2)--(3,5.8);
\draw[ultra thick, <-](-0.2, 3.2)--(-2.8, 5.8);
\draw[ultra thick, <-](0.2,3.2)--(2.8,5.8);
\draw[ultra thick, <-](-2.8,6.2)--(-0.2,8.8);
\draw[ultra thick, <-](2.8,6.2)--(0.2,8.8);
\draw[ultra thick, <-](0,6.2)--(0,8.8);
\end{tikzpicture}
\caption{$\underline{M}$} 
\label{fig:BCM}
\end{figure}
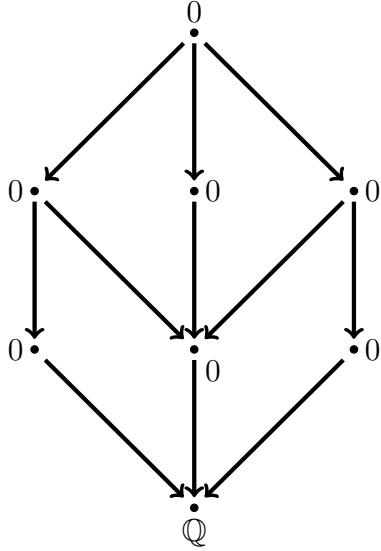

By Theorem \ref{injective envelope}, $\underline{M}$ admits an injective resolution 
$$0\rightarrow \underline{M}\rightarrow \underline{\Q}\rightarrow \underline{I_1} \rightarrow \underline{I_2}\rightarrow 0$$ 
The two new invectives $\underline{I_1}$ and $\underline{I_2}$ is given by Figure \ref{fig:I1RM} and Figure \ref{fig:I2RM}.

\begin{figure}[h]
\centering
\begin{tikzpicture}[scale=0.7]
\filldraw[black] (0,0) circle (2pt) node[below] {$0$};
\filldraw[black] (0,3) circle (2pt) node[below right] {$\Q$};
\filldraw[black] (-3,3) circle (2pt) node[left] {$\Q$};
\filldraw[black] (3,3) circle (2pt) node[right] {$\Q$};
\filldraw[black] (0,6) circle (2pt) node[right] {$\Q$};
\filldraw[black] (-3,6) circle (2pt) node[left] {$\Q^2$};
\filldraw[black] (3,6) circle (2pt) node[right] {$\Q^2$};
\filldraw[black] (0,9) circle (2pt) node[above] {$\Q^3$};
\draw[ultra thick, <-](0,0.2)--(0,2.8);
\draw[ultra thick, <-](-0.2,0.2)--(-2.8,2.8);
\draw[ultra thick, <-](0.2,0.2)--(2.8,2.8);
\draw[ultra thick, <-](0,3.2)--(0,5.8);
\draw[ultra thick, <-](-3,3.2)--(-3,5.8);
\draw[ultra thick, <-](3,3.2)--(3,5.8);
\draw[ultra thick, <-](-0.2, 3.2)--(-2.8, 5.8);
\draw[ultra thick, <-](0.2,3.2)--(2.8,5.8);
\draw[ultra thick, <-](-2.8,6.2)--(-0.2,8.8);
\draw[ultra thick, <-](2.8,6.2)--(0.2,8.8);
\draw[ultra thick, <-](0,6.2)--(0,8.8);
\node[above left] at (-1.5,7.5) {$\left(
\begin{array}{ccc}
1 &0 & 0\\
0 &1 & 0
\end{array}\right)$};
\node[above right] at (1.5,7.5) {$\left(
\begin{array}{ccc}
0 &1 & 0\\
0 &0 & 1
\end{array}\right)$};
\node[left] at (-3,4.5) {$(1,0)$};
\node[right] at (3,4.5) {$(0,1)$};
\node[above] at (-1.5,4.5) {$(0,1)$};
\node[above] at (1.5,4.5) {$(1,0)$};
\node[left] at (0,4.5) {$1$};
\node[] at (-0.5,7.5) {$(0,1,0)$};
\end{tikzpicture}
\caption{$\underline{I_1}$} 
\label{fig:I1RM}
\end{figure}

\begin{figure}[h]
\centering
\begin{tikzpicture}[scale=0.7]
\filldraw[black] (0,0) circle (2pt) node[below] {$0$};
\filldraw[black] (0,3) circle (2pt) node[below right] {$0$};
\filldraw[black] (-3,3) circle (2pt) node[left] {$0$};
\filldraw[black] (3,3) circle (2pt) node[right] {$0$};
\filldraw[black] (0,6) circle (2pt) node[right] {$0$};
\filldraw[black] (-3,6) circle (2pt) node[left] {$\Q$};
\filldraw[black] (3,6) circle (2pt) node[right] {$\Q$};
\filldraw[black] (0,9) circle (2pt) node[above] {$\Q^2$};
\draw[ultra thick, <-](0,0.2)--(0,2.8);
\draw[ultra thick, <-](-0.2,0.2)--(-2.8,2.8);
\draw[ultra thick, <-](0.2,0.2)--(2.8,2.8);
\draw[ultra thick, <-](0,3.2)--(0,5.8);
\draw[ultra thick, <-](-3,3.2)--(-3,5.8);
\draw[ultra thick, <-](3,3.2)--(3,5.8);
\draw[ultra thick, <-](-0.2, 3.2)--(-2.8, 5.8);
\draw[ultra thick, <-](0.2,3.2)--(2.8,5.8);
\draw[ultra thick, <-](-2.8,6.2)--(-0.2,8.8);
\draw[ultra thick, <-](2.8,6.2)--(0.2,8.8);
\draw[ultra thick, <-](0,6.2)--(0,8.8);
\node[above left] at (-1.5,7.5) {$(1,0)$};
\node[above right] at (1.5,7.5) {$(0,1)$};
\end{tikzpicture}
\caption{$\underline{I_2}$} 
\label{fig:I2RM}
\end{figure}
We are going to compute the $E_2$-page of the universal coefficient spectral sequence $$E_2^{p,q}=\Ext_{\cg^R}^{p,q}(\underline{H_*}(Z(K, (D^1, S^0))), \underline{M})$$
Notice that we have just computed the $q=0$ row in Section 4.1. 

\subsubsection{\underline{$q=1$ row}} The map between $\underline{H_0}(Z(K;(D^1,S^0)))$ and $\underline{I_1}$ is given in Figure \ref{D_8H_0I1}.

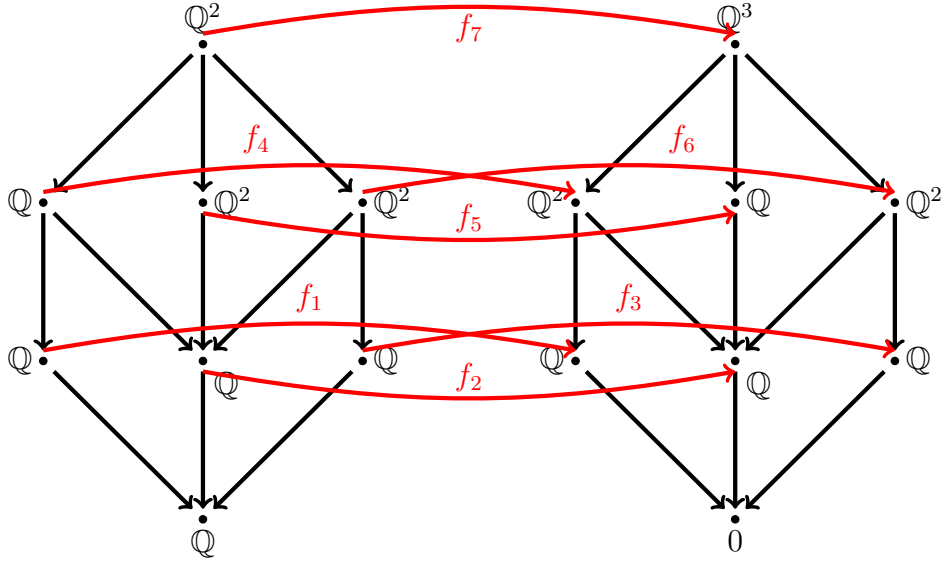
\begin{figure}[h]
\centering
\begin{tikzpicture}[scale=0.7]
\filldraw[black] (0,0) circle (2pt) node[below] {$\Q$};
\filldraw[black] (0,3) circle (2pt) node[below right] {$\Q$};
\filldraw[black] (-3,3) circle (2pt) node[left] {$\Q$};
\filldraw[black] (3,3) circle (2pt) node[right] {$\Q$};
\filldraw[black] (0,6) circle (2pt) node[right] {$\Q^2$};
\filldraw[black] (-3,6) circle (2pt) node[left] {$\Q$};
\filldraw[black] (3,6) circle (2pt) node[right] {$\Q^2$};
\filldraw[black] (0,9) circle (2pt) node[above] {$\Q^2$};
\draw[ultra thick, <-](0,0.2)--(0,2.8);
\draw[ultra thick, <-](-0.2,0.2)--(-2.8,2.8);
\draw[ultra thick, <-](0.2,0.2)--(2.8,2.8);
\draw[ultra thick, <-](0,3.2)--(0,5.8);
\draw[ultra thick, <-](-3,3.2)--(-3,5.8);
\draw[ultra thick, <-](3,3.2)--(3,5.8);
\draw[ultra thick, <-](-0.2, 3.2)--(-2.8, 5.8);
\draw[ultra thick, <-](0.2,3.2)--(2.8,5.8);
\draw[ultra thick, <-](-2.8,6.2)--(-0.2,8.8);
\draw[ultra thick, <-](2.8,6.2)--(0.2,8.8);
\draw[ultra thick, <-](0,6.2)--(0,8.8);
\filldraw[black] (10,0) circle (2pt) node[below] {$0$};
\filldraw[black] (10,3) circle (2pt) node[below right] {$\Q$};
\filldraw[black] (7,3) circle (2pt) node[left] {$\Q$};
\filldraw[black] (13,3) circle (2pt) node[right] {$\Q$};
\filldraw[black] (10,6) circle (2pt) node[right] {$\Q$};
\filldraw[black] (7,6) circle (2pt) node[left] {$\Q^2$};
\filldraw[black] (13,6) circle (2pt) node[right] {$\Q^2$};
\filldraw[black] (10,9) circle (2pt) node[above] {$\Q^3$};
\draw[ultra thick, <-](10,0.2)--(10,2.8);
\draw[ultra thick, <-](9.8,0.2)--(7.2,2.8);
\draw[ultra thick, <-](10.2,0.2)--(12.8,2.8);
\draw[ultra thick, <-](10,3.2)--(10,5.8);
\draw[ultra thick, <-](7,3.2)--(7,5.8);
\draw[ultra thick, <-](13,3.2)--(13,5.8);
\draw[ultra thick, <-](9.8, 3.2)--(7.2, 5.8);
\draw[ultra thick, <-](10.2,3.2)--(12.8,5.8);
\draw[ultra thick, <-](7.2,6.2)--(9.8,8.8);
\draw[ultra thick, <-](12.8,6.2)--(10.2,8.8);
\draw[ultra thick, <-](10,6.2)--(10,8.8);
\draw [ultra thick, ->,red] (0,2.8) to [out=-10,in=-170] (10,2.8);
\draw [ultra thick, ->,red] (0,5.8) to [out=-10,in=-170] (10,5.8);
\draw [ultra thick, ->,red] (0,9.2) to [out=10,in=170] (10,9.2);
\draw [ultra thick, ->,red] (-3,3.2) to [out=10,in=170] (7,3.2);
\draw [ultra thick, ->,red] (3,3.2) to [out=10,in=170] (13,3.2);
\draw [ultra thick, ->,red] (-3,6.2) to [out=10,in=170] (7,6.2);
\draw [ultra thick, ->,red] (3,6.2) to [out=10,in=170] (13,6.2);
\node[red] at (2,4.2) {$f_1$};
\node[red] at (5,2.7) {$f_2$};
\node[red] at (8,4.2) {$f_3$};
\node[red] at (1,7.2) {$f_4$};
\node[red] at (5,5.7) {$f_5$};
\node[red] at (9,7.2) {$f_6$};
\node[red] at (5,9.3) {$f_7$};
\end{tikzpicture}
\caption{$\Hom_{\cg^{\Q}}(\underline{H_0}(Z(K;(D^1, S^0)), \underline{I_1})$} 
\label{D_8H_0I1}
\end{figure}
\begin{enumerate}
\item
We first assume $f_1=s$, $f_2=t$, $f_3=r$.
\item
Next we set $f_4=(a_1, a_2)^T$, then by the commutative diagram
\begin{figure}[H]
\centering
\begin{tikzpicture}[scale=1]
\node[] at (0,0) {$\Q$};
\node[] at (0,2) {$\Q$};
\node[] at (2,0) {$\Q$};
\node[] at (2,2) {$\Q^2$};
\node[above] at (1,2) {$(a_1,a_2)^T$};
\node[right] at (2,1) {$(1,0)$};
\node[above] at (1,0) {$f_1=s$};
\node[left] at (0,1) {$1$};
\draw[ultra thick, ->](0,1.7)--(0,0.3);
\draw[ultra thick, ->](2,1.7)--(2,0.3);
\draw[ultra thick, ->](0.3,0)--(1.7,0);
\draw[ultra thick, ->](0.3,2)--(1.7,2);
\end{tikzpicture}
\end{figure}
we have $a_1=s$. And by another commutative diagram
\begin{figure}[H]
\centering
\begin{tikzpicture}[scale=1]
\node[] at (0,0) {$\Q$};
\node[] at (0,2) {$\Q$};
\node[] at (2,0) {$\Q$};
\node[] at (2,2) {$\Q^2$};
\node[above] at (1,2) {$(a_1,a_2)^T$};
\node[right] at (2,1) {$(0,1)$};
\node[above] at (1,0) {$f_2=t$};
\node[left] at (0,1) {$1$};
\draw[ultra thick, ->](0,1.7)--(0,0.3);
\draw[ultra thick, ->](2,1.7)--(2,0.3);
\draw[ultra thick, ->](0.3,0)--(1.7,0);
\draw[ultra thick, ->](0.3,2)--(1.7,2);
\end{tikzpicture}
\end{figure}
we have $a_2=t$. Hence $f_4=(s,t)^T$.
\item
Next consider $f_5$, by the commutative diagram
\begin{figure}[H]
\begin{tikzpicture}[scale=1]
\node[] at (0,0) {$\Q$};
\node[] at (0,2) {$\Q^2$};
\node[] at (2,0) {$\Q$};
\node[] at (2,2) {$\Q$};
\node[above] at (1,2) {$f_5$};
\node[right] at (2,1) {$1$};
\node[above] at (1,0) {$f_2=t$};
\node[left] at (0,1) {$(1,-1)$};
\draw[ultra thick, ->](0,1.7)--(0,0.3);
\draw[ultra thick, ->](2,1.7)--(2,0.3);
\draw[ultra thick, ->](0.3,0)--(1.7,0);
\draw[ultra thick, ->](0.3,2)--(1.7,2);
\end{tikzpicture}
\end{figure}
we have $f_5=(t,-t)$.
\item
For $f_6$, assume that $$f_6=\left(\begin{array}{cc}
b_1& b_2\\
b_3 &b_4
\end{array}\right).$$
From the commutative diagram
\begin{figure}[H]
\begin{tikzpicture}[scale=1]
\node[] at (0,0) {$\Q$};
\node[] at (0,2) {$\Q^2$};
\node[] at (2,0) {$\Q$};
\node[] at (2,2) {$\Q^2$};
\node[above] at (1,2) {$f_6$};
\node[right] at (2,1) {$(1,0)$};
\node[above] at (1,0) {$f_2=t$};
\node[left] at (0,1) {$(1,-1)$};
\draw[ultra thick, ->](0,1.7)--(0,0.3);
\draw[ultra thick, ->](2,1.7)--(2,0.3);
\draw[ultra thick, ->](0.3,0)--(1.7,0);
\draw[ultra thick, ->](0.3,2)--(1.7,2);
\end{tikzpicture}
\end{figure}
we conclude that $b_1=t$ and $b_2=-t$. Similarly $b_3=r$ and $b_4=-r$. Hence
$$f_6=\left(\begin{array}{cc}
t& -t\\
r & -r
\end{array}\right).$$
\item
Finally for $f_7$, assume that $$f_7=\left(\begin{array}{cc}
c_1& c_2\\
c_3 &c_4\\
c_5 &c_6
\end{array}\right).$$
From the commutative diagram
\begin{figure}[H]
\begin{tikzpicture}[scale=1]
\node[] at (0,0) {$\Q$};
\node[] at (0,2) {$\Q^2$};
\node[] at (2,0) {$\Q^2$};
\node[] at (2,2) {$\Q^3$};
\node[above] at (1,2) {$f_7$};
\node[right] at (2,1) {$\left(
\begin{array}{ccc}
1 &0 & 0\\
0 &1 & 0
\end{array}\right)$};
\node[above] at (1,0) {$f_4$};
\node[left] at (0,1) {$(1,-1)$};
\draw[ultra thick, ->](0,1.7)--(0,0.3);
\draw[ultra thick, ->](2,1.7)--(2,0.3);
\draw[ultra thick, ->](0.3,0)--(1.7,0);
\draw[ultra thick, ->](0.3,2)--(1.7,2);
\end{tikzpicture}
\end{figure}
we have $c_1=s$, $c_2=-s$, $c_3=t$ and $c_4=-t$. From the commutative diagram
\begin{figure}[H]
\begin{tikzpicture}[scale=1]
\node[] at (0,0) {$\Q^2$};
\node[] at (0,2) {$\Q^2$};
\node[] at (2,0) {$\Q^2$};
\node[] at (2,2) {$\Q^3$};
\node[above] at (1,2) {$f_7$};
\node[right] at (2,1) {$\left(
\begin{array}{ccc}
0&1 &0\\
0&0 &1
\end{array}\right)$};
\node[above] at (1,0) {$f_6$};
\node[left] at (0,1) {Id};
\draw[ultra thick, ->](0,1.7)--(0,0.3);
\draw[ultra thick, ->](2,1.7)--(2,0.3);
\draw[ultra thick, ->](0.3,0)--(1.7,0);
\draw[ultra thick, ->](0.3,2)--(1.7,2);
\end{tikzpicture}
\end{figure}
we have $c_5=r, c_6=-r$. Hence
$$f_7=\left(\begin{array}{cc}
s& -s\\
t &-t\\
r &-r
\end{array}\right).$$
\item
In summary, there are three free variables in total and $$\Hom_{\cg^{\Q}}(\underline{H_0}(Z(K;(D^1, S^0)), \underline{I_1})=\Q^3.$$
\end{enumerate}
The map between $\underline{H_1}(Z(K;(D^1,S^0)))$ and $\underline{I_1}$ is given in Figure \ref{D_8H_1I1}.

\begin{figure}[h]
\centering
\begin{tikzpicture}[scale=0.7]
\filldraw[black] (0,0) circle (2pt) node[below] {$0$};
\filldraw[black] (0,3) circle (2pt) node[below right] {$\Q$};
\filldraw[black] (-3,3) circle (2pt) node[left] {$0$};
\filldraw[black] (3,3) circle (2pt) node[right] {$\Q$};
\filldraw[black] (0,6) circle (2pt) node[right] {$0$};
\filldraw[black] (-3,6) circle (2pt) node[left] {$\Q$};
\filldraw[black] (3,6) circle (2pt) node[right] {$0$};
\filldraw[black] (0,9) circle (2pt) node[above] {$0$};
\draw[ultra thick, <-](0,0.2)--(0,2.8);
\draw[ultra thick, <-](-0.2,0.2)--(-2.8,2.8);
\draw[ultra thick, <-](0.2,0.2)--(2.8,2.8);
\draw[ultra thick, <-](0,3.2)--(0,5.8);
\draw[ultra thick, <-](-3,3.2)--(-3,5.8);
\draw[ultra thick, <-](3,3.2)--(3,5.8);
\draw[ultra thick, <-](-0.2, 3.2)--(-2.8, 5.8);
\draw[ultra thick, <-](0.2,3.2)--(2.8,5.8);
\draw[ultra thick, <-](-2.8,6.2)--(-0.2,8.8);
\draw[ultra thick, <-](2.8,6.2)--(0.2,8.8);
\draw[ultra thick, <-](0,6.2)--(0,8.8);
\filldraw[black] (10,0) circle (2pt) node[below] {$0$};
\filldraw[black] (10,3) circle (2pt) node[below right] {$\Q$};
\filldraw[black] (7,3) circle (2pt) node[left] {$\Q$};
\filldraw[black] (13,3) circle (2pt) node[right] {$\Q$};
\filldraw[black] (10,6) circle (2pt) node[right] {$\Q$};
\filldraw[black] (7,6) circle (2pt) node[left] {$\Q^2$};
\filldraw[black] (13,6) circle (2pt) node[right] {$\Q^2$};
\filldraw[black] (10,9) circle (2pt) node[above] {$\Q^3$};
\draw[ultra thick, <-](10,0.2)--(10,2.8);
\draw[ultra thick, <-](9.8,0.2)--(7.2,2.8);
\draw[ultra thick, <-](10.2,0.2)--(12.8,2.8);
\draw[ultra thick, <-](10,3.2)--(10,5.8);
\draw[ultra thick, <-](7,3.2)--(7,5.8);
\draw[ultra thick, <-](13,3.2)--(13,5.8);
\draw[ultra thick, <-](9.8, 3.2)--(7.2, 5.8);
\draw[ultra thick, <-](10.2,3.2)--(12.8,5.8);
\draw[ultra thick, <-](7.2,6.2)--(9.8,8.8);
\draw[ultra thick, <-](12.8,6.2)--(10.2,8.8);
\draw[ultra thick, <-](10,6.2)--(10,8.8);
\draw [ultra thick, ->,red] (0,2.8) to [out=-10,in=-170] (10,2.8);
\draw [ultra thick, ->,red] (3,3.2) to [out=10,in=170] (13,3.2);
\draw [ultra thick, ->,red] (-3,6.2) to [out=10,in=170] (7,6.2);
\node[red] at (5,2.7) {$f_2$};
\node[red] at (8,4.2) {$f_3$};
\node[red] at (1,7.2) {$f_4$};
\end{tikzpicture}
\caption{$\Hom_{\cg^{\Q}}(\underline{H_1}(Z(K;(D^1, S^0)), \underline{I_1})$} 
\label{D_8H_1I1}
\end{figure}
\begin{enumerate}
\item
First we assume $f_2=s$, $f_3=t$ and $f_4=(a,b)^T$
\item
From the following two commutative diagrams
\begin{figure}[H]
\begin{tikzpicture}[scale=1]
\node[] at (0,0) {$0$};
\node[] at (0,2) {$\Q$};
\node[] at (2,0) {$\Q$};
\node[] at (2,2) {$\Q^2$};
\node[above] at (1,2) {$f_4$};
\node[right] at (2,1) {$(1,0)$};
\node[above] at (1,0) {};
\node[left] at (0,1) {};
\draw[ultra thick, ->](0,1.7)--(0,0.3);
\draw[ultra thick, ->](2,1.7)--(2,0.3);
\draw[ultra thick, ->](0.3,0)--(1.7,0);
\draw[ultra thick, ->](0.3,2)--(1.7,2);
\end{tikzpicture}
\begin{tikzpicture}[scale=1]
\node[] at (0,0) {$\Q$};
\node[] at (0,2) {$\Q$};
\node[] at (2,0) {$\Q$};
\node[] at (2,2) {$\Q^2$};
\node[above] at (1,2) {$f_4$};
\node[right] at (2,1) {$(0,1)$};
\node[above] at (1,0) {$f_2=s$};
\node[left] at (0,1) {$1$};
\draw[ultra thick, ->](0,1.7)--(0,0.3);
\draw[ultra thick, ->](2,1.7)--(2,0.3);
\draw[ultra thick, ->](0.3,0)--(1.7,0);
\draw[ultra thick, ->](0.3,2)--(1.7,2);
\end{tikzpicture}
\end{figure}
we conclude that $f_4=(0,s)^T$. 
\item
In summary, $$\Hom_{\cg^{\Q}}(\underline{H_1}(Z(K;(D^1, S^0)), \underline{I_1})=\Q^2.$$
\end{enumerate}

The map between $\underline{H_2}(Z(K;(D^1,S^0)))$ and $\underline{I_1}$ is given in Figure \ref{D_8H_2I1}.

\begin{figure}[h]
\centering
\begin{tikzpicture}[scale=0.7]
\filldraw[black] (0,0) circle (2pt) node[below] {$0$};
\filldraw[black] (0,3) circle (2pt) node[below right] {$0$};
\filldraw[black] (-3,3) circle (2pt) node[left] {$\Q$};
\filldraw[black] (3,3) circle (2pt) node[right] {$0$};
\filldraw[black] (0,6) circle (2pt) node[right] {$0$};
\filldraw[black] (-3,6) circle (2pt) node[left] {$0$};
\filldraw[black] (3,6) circle (2pt) node[right] {$0$};
\filldraw[black] (0,9) circle (2pt) node[above] {$0$};
\draw[ultra thick, <-](0,0.2)--(0,2.8);
\draw[ultra thick, <-](-0.2,0.2)--(-2.8,2.8);
\draw[ultra thick, <-](0.2,0.2)--(2.8,2.8);
\draw[ultra thick, <-](0,3.2)--(0,5.8);
\draw[ultra thick, <-](-3,3.2)--(-3,5.8);
\draw[ultra thick, <-](3,3.2)--(3,5.8);
\draw[ultra thick, <-](-0.2, 3.2)--(-2.8, 5.8);
\draw[ultra thick, <-](0.2,3.2)--(2.8,5.8);
\draw[ultra thick, <-](-2.8,6.2)--(-0.2,8.8);
\draw[ultra thick, <-](2.8,6.2)--(0.2,8.8);
\draw[ultra thick, <-](0,6.2)--(0,8.8);
\filldraw[black] (10,0) circle (2pt) node[below] {$0$};
\filldraw[black] (10,3) circle (2pt) node[below right] {$\Q$};
\filldraw[black] (7,3) circle (2pt) node[left] {$\Q$};
\filldraw[black] (13,3) circle (2pt) node[right] {$\Q$};
\filldraw[black] (10,6) circle (2pt) node[right] {$\Q$};
\filldraw[black] (7,6) circle (2pt) node[left] {$\Q^2$};
\filldraw[black] (13,6) circle (2pt) node[right] {$\Q^2$};
\filldraw[black] (10,9) circle (2pt) node[above] {$\Q^3$};
\draw[ultra thick, <-](10,0.2)--(10,2.8);
\draw[ultra thick, <-](9.8,0.2)--(7.2,2.8);
\draw[ultra thick, <-](10.2,0.2)--(12.8,2.8);
\draw[ultra thick, <-](10,3.2)--(10,5.8);
\draw[ultra thick, <-](7,3.2)--(7,5.8);
\draw[ultra thick, <-](13,3.2)--(13,5.8);
\draw[ultra thick, <-](9.8, 3.2)--(7.2, 5.8);
\draw[ultra thick, <-](10.2,3.2)--(12.8,5.8);
\draw[ultra thick, <-](7.2,6.2)--(9.8,8.8);
\draw[ultra thick, <-](12.8,6.2)--(10.2,8.8);
\draw[ultra thick, <-](10,6.2)--(10,8.8);
\draw [ultra thick, ->,red] (-3,3.2) to [out=10,in=170] (7,3.2);
\node[red] at (2,4.2) {$f_1$};
\end{tikzpicture}
\caption{$\Hom_{\cg^{\Q}}(\underline{H_2}(Z(K;(D^1, S^0)), \underline{I_1})$} 
\label{D_8H_2I1}
\end{figure}
Hence $$\Hom_{\cg^{\Q}}(\underline{H_2}(Z(K;(D^1, S^0)), \underline{I_1})=\Q.$$
Moreover, we have $$\Hom_{\cg^{\Q}}(\underline{H_3}(Z(K;(D^1, S^0)), \underline{I_1})=0.$$
because there is no possible nonzero maps.
\subsubsection{\underline{$q=2$ row}}
The map between $\underline{H_0}(Z(K;(D^1,S^0)))$ and $\underline{I_2}$ is given in Figure \ref{D_8H_0I2}.

\begin{figure}[h]
\centering
\begin{tikzpicture}[scale=0.7]
\filldraw[black] (0,0) circle (2pt) node[below] {$\Q$};
\filldraw[black] (0,3) circle (2pt) node[below right] {$\Q$};
\filldraw[black] (-3,3) circle (2pt) node[left] {$\Q$};
\filldraw[black] (3,3) circle (2pt) node[right] {$\Q$};
\filldraw[black] (0,6) circle (2pt) node[right] {$\Q^2$};
\filldraw[black] (-3,6) circle (2pt) node[left] {$\Q$};
\filldraw[black] (3,6) circle (2pt) node[right] {$\Q^2$};
\filldraw[black] (0,9) circle (2pt) node[above] {$\Q^2$};
\draw[ultra thick, <-](0,0.2)--(0,2.8);
\draw[ultra thick, <-](-0.2,0.2)--(-2.8,2.8);
\draw[ultra thick, <-](0.2,0.2)--(2.8,2.8);
\draw[ultra thick, <-](0,3.2)--(0,5.8);
\draw[ultra thick, <-](-3,3.2)--(-3,5.8);
\draw[ultra thick, <-](3,3.2)--(3,5.8);
\draw[ultra thick, <-](-0.2, 3.2)--(-2.8, 5.8);
\draw[ultra thick, <-](0.2,3.2)--(2.8,5.8);
\draw[ultra thick, <-](-2.8,6.2)--(-0.2,8.8);
\draw[ultra thick, <-](2.8,6.2)--(0.2,8.8);
\draw[ultra thick, <-](0,6.2)--(0,8.8);
\filldraw[black] (10,0) circle (2pt) node[below] {$0$};
\filldraw[black] (10,3) circle (2pt) node[below right] {$0$};
\filldraw[black] (7,3) circle (2pt) node[left] {$0$};
\filldraw[black] (13,3) circle (2pt) node[right] {$0$};
\filldraw[black] (10,6) circle (2pt) node[right] {$0$};
\filldraw[black] (7,6) circle (2pt) node[left] {$\Q$};
\filldraw[black] (13,6) circle (2pt) node[right] {$\Q$};
\filldraw[black] (10,9) circle (2pt) node[above] {$\Q^2$};
\draw[ultra thick, <-](10,0.2)--(10,2.8);
\draw[ultra thick, <-](9.8,0.2)--(7.2,2.8);
\draw[ultra thick, <-](10.2,0.2)--(12.8,2.8);
\draw[ultra thick, <-](10,3.2)--(10,5.8);
\draw[ultra thick, <-](7,3.2)--(7,5.8);
\draw[ultra thick, <-](13,3.2)--(13,5.8);
\draw[ultra thick, <-](9.8, 3.2)--(7.2, 5.8);
\draw[ultra thick, <-](10.2,3.2)--(12.8,5.8);
\draw[ultra thick, <-](7.2,6.2)--(9.8,8.8);
\draw[ultra thick, <-](12.8,6.2)--(10.2,8.8);
\draw[ultra thick, <-](10,6.2)--(10,8.8);
\draw [ultra thick, ->,red] (0,9.2) to [out=10,in=170] (10,9.2);
\draw [ultra thick, ->,red] (-3,6.2) to [out=10,in=170] (7,6.2);
\draw [ultra thick, ->,red] (3,6.2) to [out=10,in=170] (13,6.2);
\node[red] at (1,7.2) {$f_4$};
\node[red] at (9,7.2) {$f_6$};
\node[red] at (5,9.3) {$f_7$};
\end{tikzpicture}
\caption{$\Hom_{\cg^{\Q}}(\underline{H_0}(Z(K;(D^1, S^0)), \underline{I_2})$} 
\label{D_8H_0I2}
\end{figure}
\begin{enumerate}
\item
Firstly we assume that $f_4=r$ and $f_6=(s,t)$.
\item
From the following two commutative diagrams
\begin{figure}[H]
\begin{tikzpicture}[scale=1]
\node[] at (0,0) {$\Q$};
\node[] at (0,2) {$\Q^2$};
\node[] at (2,0) {$\Q$};
\node[] at (2,2) {$\Q^2$};
\node[above] at (1,2) {$f_7$};
\node[right] at (2,1) {$(1,0)$};
\node[above] at (1,0) {$r$};
\node[left] at (0,1) {$(1,-1)$};
\draw[ultra thick, ->](0,1.7)--(0,0.3);
\draw[ultra thick, ->](2,1.7)--(2,0.3);
\draw[ultra thick, ->](0.3,0)--(1.7,0);
\draw[ultra thick, ->](0.3,2)--(1.7,2);
\end{tikzpicture}
\begin{tikzpicture}[scale=1]
\node[] at (0,0) {$\Q^2$};
\node[] at (0,2) {$\Q^2$};
\node[] at (2,0) {$\Q$};
\node[] at (2,2) {$\Q^2$};
\node[above] at (1,2) {$f_7$};
\node[right] at (2,1) {$(0,1)$};
\node[above] at (1,0) {$(s,t)$};
\node[left] at (0,1) {Id};
\draw[ultra thick, ->](0,1.7)--(0,0.3);
\draw[ultra thick, ->](2,1.7)--(2,0.3);
\draw[ultra thick, ->](0.3,0)--(1.7,0);
\draw[ultra thick, ->](0.3,2)--(1.7,2);
\end{tikzpicture}
\end{figure}
We have $$f_7=\left(\begin{array}{cc}
r& -r\\
s &t\\
\end{array}\right).$$
\item
In summary, $$\Hom_{\cg^{\Q}}(\underline{H_0}(Z(K;(D^1, S^0)), \underline{I_2})=\Q^3.$$
\end{enumerate}
The map between $\underline{H_1}(Z(K;(D^1,S^0)))$ and $\underline{I_2}$ is given in Figure \ref{D_8H_1I2}.

\begin{figure}[h]
\centering
\begin{tikzpicture}[scale=0.7]
\filldraw[black] (0,0) circle (2pt) node[below] {$0$};
\filldraw[black] (0,3) circle (2pt) node[below right] {$\Q$};
\filldraw[black] (-3,3) circle (2pt) node[left] {$0$};
\filldraw[black] (3,3) circle (2pt) node[right] {$\Q$};
\filldraw[black] (0,6) circle (2pt) node[right] {$0$};
\filldraw[black] (-3,6) circle (2pt) node[left] {$\Q$};
\filldraw[black] (3,6) circle (2pt) node[right] {$0$};
\filldraw[black] (0,9) circle (2pt) node[above] {$0$};
\draw[ultra thick, <-](0,0.2)--(0,2.8);
\draw[ultra thick, <-](-0.2,0.2)--(-2.8,2.8);
\draw[ultra thick, <-](0.2,0.2)--(2.8,2.8);
\draw[ultra thick, <-](0,3.2)--(0,5.8);
\draw[ultra thick, <-](-3,3.2)--(-3,5.8);
\draw[ultra thick, <-](3,3.2)--(3,5.8);
\draw[ultra thick, <-](-0.2, 3.2)--(-2.8, 5.8);
\draw[ultra thick, <-](0.2,3.2)--(2.8,5.8);
\draw[ultra thick, <-](-2.8,6.2)--(-0.2,8.8);
\draw[ultra thick, <-](2.8,6.2)--(0.2,8.8);
\draw[ultra thick, <-](0,6.2)--(0,8.8);
\filldraw[black] (10,0) circle (2pt) node[below] {$0$};
\filldraw[black] (10,3) circle (2pt) node[below right] {$0$};
\filldraw[black] (7,3) circle (2pt) node[left] {$0$};
\filldraw[black] (13,3) circle (2pt) node[right] {$0$};
\filldraw[black] (10,6) circle (2pt) node[right] {$0$};
\filldraw[black] (7,6) circle (2pt) node[left] {$\Q$};
\filldraw[black] (13,6) circle (2pt) node[right] {$\Q$};
\filldraw[black] (10,9) circle (2pt) node[above] {$\Q^2$};
\draw[ultra thick, <-](10,0.2)--(10,2.8);
\draw[ultra thick, <-](9.8,0.2)--(7.2,2.8);
\draw[ultra thick, <-](10.2,0.2)--(12.8,2.8);
\draw[ultra thick, <-](10,3.2)--(10,5.8);
\draw[ultra thick, <-](7,3.2)--(7,5.8);
\draw[ultra thick, <-](13,3.2)--(13,5.8);
\draw[ultra thick, <-](9.8, 3.2)--(7.2, 5.8);
\draw[ultra thick, <-](10.2,3.2)--(12.8,5.8);
\draw[ultra thick, <-](7.2,6.2)--(9.8,8.8);
\draw[ultra thick, <-](12.8,6.2)--(10.2,8.8);
\draw[ultra thick, <-](10,6.2)--(10,8.8);
\draw [ultra thick, ->,red] (-3,6.2) to [out=10,in=170] (7,6.2);
\node[red] at (1,7.2) {$f_4$};
\end{tikzpicture}
\caption{$\Hom_{\cg^{\Q}}(\underline{H_1}(Z(K;(D^1, S^0)), \underline{I_2})$} 
\label{D_8H_1I2}
\end{figure}
Hence $$\Hom_{\cg^{\Q}}(\underline{H_1}(Z(K;(D^1, S^0)), \underline{I_2})=\Q.$$
Finally we have $$\Hom_{\cg^{\Q}}(\underline{H_2}(Z(K;(D^1, S^0)), \underline{I_2})= \Hom_{\cg^{\Q}}(\underline{H_3}(Z(K;(D^1, S^0)), \underline{I_2})=0.$$
because there is no possible nonzero maps.
Combine the computation in this section, the $E_2$-page of the universal spectral sequence 
$$E_2^{p,q}=\Ext_{\cg^R}^{p,q}(\underline{H_*}(Z(K, (D^1, S^0))), \underline{M})$$
is given by the follow diagram. We only write the dimension for each $\Q$-vector spaces.
\begin{figure}[H]
\centering
\begin{tikzpicture}[scale=0.7]
\node[] at (1,1) {2};
\node[] at (2,1) {0};
\node[] at (3,1) {0};
\node[] at (4,1) {1};
\node[] at (1,2) {3};
\node[] at (2,2) {2};
\node[] at (3,2) {1};
\node[] at (4,2) {0};
\node[] at (1,3) {3};
\node[] at (2,3) {1};
\node[] at (3,3) {0};
\node[] at (4,3) {0};
\draw[ultra thick, -](-1,0)--(5,0);
\draw[ultra thick, -](0,-1)--(0,4);
\end{tikzpicture}
\caption{$\Ext_{\cg^R}^{p,q}(\underline{H_*}(Z(K, (D^1, S^0))), \underline{M})$} 
\label{E_2page}
\end{figure}

Notice that there are two possible $d_2$ differentials and one possible $d_3$ differential. In addition, we have $$H^1_{G}(Z(K;(D^1, S^0)),\underline{M})=Q^3\neq H^1_{G}(Z(K;(D^1, S^0)),\underline{\Q}).$$
And $$\dim_{\Q}H^2_{G}(Z(K;(D^1, S^0)),\underline{M})\geq 3$$ Hence $$H^2_{G}(Z(K;(D^1, S^0)),\underline{M})\neq H^2_{G}(Z(K;(D^1, S^0)),\underline{\Q}). $$
\nocite{*}
\bibliographystyle{abbrv}

\bibliography{BCEH}

\end{document}